\documentclass[11pt]{amsart}
\usepackage[left=2cm,right=2cm,top=2cm,bottom=2.5cm,a4paper]{geometry}   	 
\usepackage{amscd, amssymb, mathrsfs, mathabx, tikz-cd, amsthm, thmtools, stmaryrd}

\usepackage{hyperref}

\usepackage[all]{xy}

\numberwithin{equation}{section}

\newcommand\bC{\bold{C}}
\newcommand\bD{\bold{D}}
\newcommand\bE{\bold{E}}

\newcommand\bK{\bold{K}}

\newcommand\bP{\bold{P}}
\newcommand\bQ{\bold{Q}}

\newcommand{\CC}{\mathbb{C}}

\newcommand{\EE}{\mathbb{E}}

\newcommand{\LL}{\mathbb{L}}

\newcommand{\PP}{\mathbb{P}}

\newcommand{\cal}{\mathcal}

\newcommand\cA{{\cal A}}
\newcommand\cB{{\cal B}}
\newcommand\cC{{\cal C}}
\newcommand\cD{{\cal D}}

\newcommand\cH{{\cal H}}

\newcommand\cL{{\cal L}}

\newcommand\cO{{\cal O}}

\def\fC{\mathfrak{C}}

\def\fM{\mathfrak{M}}

\DeclareMathOperator{\rank}{rank}

\def\surra{\twoheadrightarrow}

\def\hra{\hookrightarrow}

\def\lra{\longrightarrow }
\def\bar{\overline}

\def\til{\tilde}
\def\wtil{\widetilde}
\def\what{\widehat}

\def\ker{\mathrm{ker}}

\def\deg{\mathrm{deg}}
\def\dim{\mathrm{dim}}

\def\rank{\mathrm{rank}}

\def\Spec{\mathrm{Spec}}

\def\c>{\succ}
\def\c<{\prec}

\def\l({\left(}
\def\r){\right)}

\def\Sym{\mathrm{Sym}}

\def\Bl{\mathrm{Bl}}

\newtheorem{Thm*}{Theorem}
\newtheorem{prop}{Proposition}[section]

\newtheorem{lemm}[prop]{Lemma}
\newtheorem{coro}[prop]{Corollary}
\newtheorem{rema}[prop]{Remark}
\newtheorem{exam}[prop]{Example}

\newtheorem{defi}[prop]{Definition}

\def\dual{^{\vee}}

\def\virt{^{\mathrm{vir}} }
\def\loc{{\mathrm{loc}} }

\def\Spec{\mathrm{Spec} }

\newcommand{\KL}{\mathrm{KL}}
\newcommand{\red}{\mathrm{red}}
\newcommand{\Def}{\mathrm{Def}}

\newcommand{\refi}{\mathrm{ref}}
\def\det{\mathrm{det}}

\title[Quantum Lefschetz property for genus 2 quasimap invariants]{Quantum Lefschetz property for genus two stable quasimap invariants} 

\author{Sanghyeon Lee}
\address{School of Mathematics, KIAS, 85 Hoegiro, Dongdaemun-gu, Seoul 02455, Korea}
\email{sanghyeon@kias.re.kr}

\author[Mu-Lin Li]{Mu-Lin Li}
\address{School of Mathematics, Hunan University, China} \email{mulin@hnu.edu.cn}

\author{Jeongseok Oh}
\address{Department of Mathematics, Imperial College, London SW7 2AZ, United Kingdom}
\email{j.oh@imperial.ac.uk}

\thanks{}

\date{}

\begin{document}

\begin{abstract}
By the {\em reduced component} in a moduli space of {\em stable quasimaps} to $n$-dimensional projective space $\PP^n$ we mean the closure of the locus in which the domain curves are smooth. As in the moduli space of stable maps, we prove the reduced component is smooth in {\em genus $2$, degree $\geq 3$}.

Then we prove the virtual fundamental cycle of the moduli space of stable quasimaps to a {\em complete intersection $X$ in $\PP^n$} of genus $2$, degree $\geq 3$ is explicitly expressed in terms of the fundamental cycle of the reduced component of $\PP^n$ and virtual cycles of lower genus $<2$ moduli spaces of $X$.
\end{abstract}

\maketitle

\section*{Introduction}
Computing Gromov--Witten invariants of the quintic $3$-fold $X$ has attracted interests of both mathematicians and physicists due to its importance in mirror symmetry, which mainly studies Calabi--Yau $3$-folds. 
One effective way to conquer this computation is to relate them with GW invariants of $\PP^4$ in which $X$ is embedded. Then we apply virtual localisation \cite{GP99} for the natural torus action on $\PP^4$ to compute them.
We will call this principle relating GW invariants of $X$ and $\PP^4$ the {\em quantum Lefschetz property}.

The name, quantum Lefschetz, is originally from the formula between genus $0$ virtual cycles: Let $\iota: M(X)\hookrightarrow M(\PP^4)$ be the moduli spaces of stable maps to $X\hookrightarrow \PP^4$, respectively. On $M(\PP^4)$ a coherent sheaf $V:=\pi_*\mathsf{f}^*\cO_{\PP^4}(5)$ is defined via the universal curve $\pi:C\to M(\PP^4)$ and the universal map $\mathsf{f}:C\to \PP^4$. In genus $0$, $M(\PP^4)$ is smooth and $V$ is a vector bundle. Then the quantum Lefschetz formula \cite{KKP03} asserts that
\begin{align}\label{naiveQLP1}
\iota_*[M(X)]\virt \ = \ e(V) \ \cap\ [M(\PP^4)].
\end{align}
Unfortunately, it turns out that \eqref{naiveQLP1} does not hold for higher genus invariants \cite{Gi98}. So we need more sophisticated version of the quantum Lefschetz property for higher genus invariants.

\smallskip
Meanwhile, the explicit relationship between GW and {\em stable quasimap invariants} of $X$ is known to be wall-crossing formula \cite{CK20, CJR21-1, Zh22}. Since we may expect a relatively simpler version of quantum Lefschetz property for higher genus quasimap invariants, wall-crossing formula allows us to study simpler quantum Lefschetz property to compute GW invariants. For instance the original quantum Lefschetz formula \eqref{naiveQLP1} holds true for genus $1$ quasimap invariants, so it dramatically helps the computation of genus $1$ GW invariants \cite{KL18}.

We notice that there has been several interesting quantum Lefschetz formulae for higher genus GW or quasimap invariants, or relationships between invariants of $X$ and other invariants, developed in a recent few years \cite{Zi1, Zi08, CZ14, CL15, KL18, CLLL16, FL19, BCM20, CM18, CJRS18, CGLL21, LO18, CGL21, CJR21-2, LO20}. These lead us some actual computations of higher genus invariants \cite{Zi2, Po13, KL18, GJR17, FL19, CGL18, GJR18}.
In our paper we would like to introduce one more quantum Lefschetz formula for genus $2$ quasimap invariants. Our formulae \eqref{qlp1}, \eqref{qlp} contain Zinger-type reduced virtual cycles, which have not been studied in any of references above for genus $\geq 2$ yet. Since it is expected to have some interesting properties -- such as integrability -- we hope our new formulae would suggest some idea in studying higher genus invariants. 

To construct Zinger-type reduced virtual cycles, we need to study the reduced components on which the cycles are supported (conjecturally), in the moduli spaces of stable maps or stable quasimaps to $\PP^n$. It is firstly addressed in \cite{VZ08,HL10} where they studied genus $1$ stable maps. Later \cite{HLN18, BC} studied genus $2$ stable maps in different ways -- \cite{HLN18} is closer to the original idea of \cite{VZ08,HL10}, whereas \cite{BC} uses curves with Gorenstein singularities. Although \cite{BC} studied more general target spaces, we follow the idea of \cite{HLN18} to construct our reduced virtual cycles due to its advantage on computations.

\smallskip
We consider a slight more general situation. Let $X = \{f_1 = \dots = f_m = 0\}$ be a complete intersection in projective space $\PP^n$, where $f_i \in \Gamma(\PP^n, \cO_{\PP^n}(\ell_i))$. When $n=4$, $m=1$ and $\ell_1=5$ it recovers a quintic threefold $X$. We denote by $Q_{g,k,d}(X)\hookrightarrow Q_{g,k,d}(\PP^n)$ the moduli spaces of stable quasimaps to $X\hookrightarrow \PP^n$ of genus $g$, degree $d$ with $k$ marked points. Using the universal curve and map
$$
\xymatrix@R=6mm{
C\ar[r]^-{\mathsf{f}} \ar[d]^-{\pi} & [\CC^{n+1}/\CC^*]\\
Q_{g,k,d}(\PP^n), & 
}
$$
we define $V_{g,k,d}:=\oplus_{i=1}^m\pi_*\mathsf{f}^*\cO(\ell_i)$, where $\cO(d):=[\CC^{n+1}\times\CC/\CC^*]$ is a bundle defined by weight $d$ representation. Let $Q^{\red}_{g,k,d}(\PP^n)$ be the closure of the open substack in $Q_{g,k,d}(\PP^n)$ on which $R^1\pi_*\mathsf{f}^*\cO(1)$ vanishes
$$
Q^{\red}_{g,k,d}(\PP^n)\ :=\ \mathrm{closure}\left(Q_{g,k,d}(\PP^n)\smallsetminus\mathrm{supp}R^1\pi_*\mathsf{f}^*\cO(1)\right)\ \subset\ Q_{g,k,d}(\PP^n).
$$ 
Then on the proper birational base change $\wtil{Q}_{g,k,d}(\PP^n)\to Q_{g,k,d}(\PP^n)$ in Section \ref{desin}, the proper transform of $Q^{\red}_{g,k,d}(\PP^n)$ is smooth and $V_{g,k,d}$ over there is a bundle. We denote by $\LL$ the tautological bundle associated to the marked point, a line bundle formed by the cotangent line at the marked point.

Then we prove the following quantum Lefschetz formula for a Calabi-Yau $3$-fold.

\begin{Thm*}\label{QLP1}
When $X$ is a Calabi-Yau $3$-fold, $d\geq 3$, we have an equivalence in the Chow group of $Q_{2,0,d}(X)$,
\begin{align}\label{qlp1}
[Q_{2,0,d}(X)]\virt = \ & e^{\refi}(V_{2,0,d})\cap [Q_{2,0,d}^{\red}(\PP^n)] \\ \nonumber
& - \frac{c_1(\LL)}{24}\cap [Q_{1,1,d}(X)]\virt \\ \nonumber
& + \frac{1}{24^2}\l(\frac{c_1(\LL_1)c_1(\LL_2)}{2}-\frac{3(\mathrm{ev}^*_1 c_2(T_X) + \mathrm{ev}^*_2 c_2(T_X))}{2} \r) \cap [Q_{0,2,d}(X)]\virt .
\end{align}
\end{Thm*}

Using the defining section $f=(f_i)_i\in \Gamma(\PP^n,\oplus_i\cO(\ell_i))$ of $X\subset \PP^n$, the first term in the RHS of \eqref{qlp1} is localised to $Q(X):=Q_{2,0,d}(X)$ via refined Euler class $e^{\mathrm{ref}}(V_{2,0,d})$ \cite[Section 14.1]{Fu}\footnote{This is called the localised top Chern class there.} defined by the section $\pi_*\mathsf{f}^*f\in \Gamma(V_{2,0,d})$ cutting out $Q(X)=(\pi_*\mathsf{f}^*f)^{-1}(0)$. The last two terms in the RHS are cycles on $Q(X)$ via the pushforwards of embeddings,
\smallskip
\begin{enumerate}
\item 
$ \iota_1 : \overline{M}_{1,1} \times Q_{1,1,d}(X) \hookrightarrow Q(X)$,
\begin{figure}[h]
\begin{center}
\includegraphics[scale=0.2]{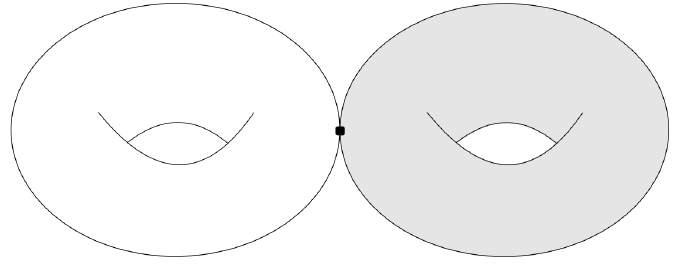}
\end{center}
\end{figure}
\\
\item
$ \iota_2 : \overline{M}_{1,1} \times Q_{0,2,d}(X) \times \overline{M}_{1,1} \xrightarrow{2:1} Q(X)$.
\begin{figure}[h]
\begin{center}
\includegraphics[scale=0.2]{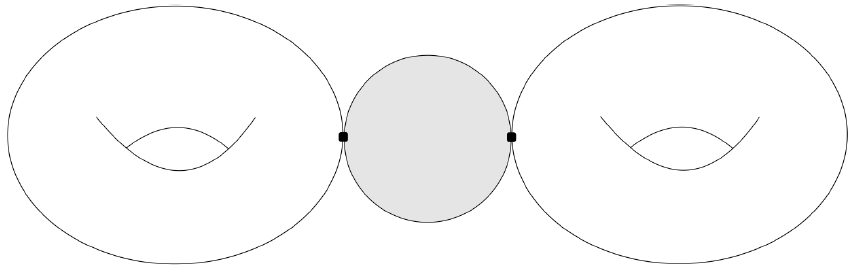}
\end{center}
\end{figure}
\\
\end{enumerate}
On these loci $\oplus_iR^1\pi_*\mathsf{f}^*\cO(\ell_i)$ does not vanish, obstructing the original formula \eqref{naiveQLP1} is satisfied. Note that the image of $\iota_2$ is contained in the image of $\iota_1$, but the rank of $\oplus_iR^1\pi_*\mathsf{f}^*\cO(\ell_i)$ jumps on the image of $\iota_2$.

\smallskip

In fact Theorem \ref{QLP1} for a Calabi-Yau $3$-fold is induced by the following quantum Lefschetz formula in Theorem \ref{QLP} for any complete intersection. 
In this general case, we may have a nontrivial contribution from 
\begin{enumerate}
\item[(3)]
$ \iota_3 : \overline{M}_{1,2} \times Q'_{0,2,d}(X) \hookrightarrow Q(X)$,
\begin{figure}[h]
\begin{center}
\includegraphics[scale=0.25]{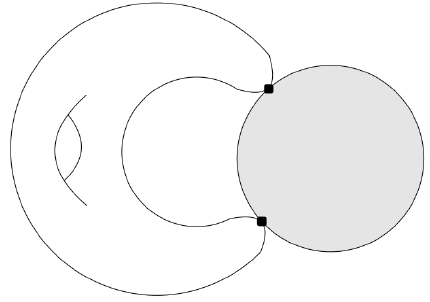}
\end{center}
\end{figure}
\\
where $Q'_{0,2,d}(X)\hookrightarrow Q_{0,2,d}(X)$ is the closed substack on which the two evaluation maps are the same $\mathrm{ev}_1=\mathrm{ev}_2$.
\end{enumerate}
as well. This (3) does not have a contribution for a Calabi-Yau $3$-fold. These three loci in (1), (2) and (3) are the places where $\oplus_iR^1\pi_*\mathsf{f}^*\cO(\ell_i)$ does not vanish exactly.

Before stating Theorem we introduce some (Chow) cohomology classes to simplify the statement. Denoting by $\cH$ the Hodge bundle $\cH:=\pi_*\omega_C$ we define
$$
K\ :=\ \frac{c\,(\cH^\vee \boxtimes \mathrm{ev}^* T_X)}{c\,(\LL^\vee \boxtimes \LL^\vee)},\ 
A^t\ :=\ \frac{c\,(\cH^\vee \boxtimes \mathrm{ev}^* T_X)}{c\,(\LL^\vee \boxtimes 1)^t},\ B\ :=\ \frac{1}{c\,(1 \boxtimes \LL^\vee)}.
$$
And we denote by $K_i$, $A^t_i$, $B_i$ the classes corresponding to the $i$-th marked point, whereas by $[K]_i$, $[A^t]_i$, $[B]_i$ the degree $i$ parts. We also define a (Chow) homology class
\begin{align}\label{Q1'}
[Q'_{0,2,d}(X)]\virt\ :=\ (\mathrm{ev}_1 \times \mathrm{ev}_2)^*\Delta_X \cap [Q_{0,2,d}(X)]\virt 
\end{align}
using the diagonal class $\Delta_X\in A^{\dim X}(X\times X)$. The bundle $V_{2,0,d}$ on $Q^{\mathrm{red}}_{2,0,d}(\PP^n)$ is defined by $\oplus_i \pi_*\mathsf{f}^*\cO(\ell_i)$.

\begin{Thm*}\label{QLP}
For $d \geq 3$, we have an equivalence in the Chow group of $Q_{2,0,d}(X)$,
\begin{align}
[Q_{2,0,d}(X)]\virt  =&  \; e^{\mathrm{ref}}(V_{2,0,d}) \cap [Q^{\mathrm{red}}_{2,0,d}(\PP^n)]  \nonumber \\
& + [K]_{\dim X-1} \cap  \left( [\overline{M}_{1,1}] \times [Q_{1,1,d}(X)]\virt \right) \nonumber  \\ \label{qlp}  
& + \left( \frac{[K_1 K_2]_{2\dim X-2}}{2}-[K_1]_{\dim X-1}[K_2]_{\dim X-1} \right) \cap \l( \, [\overline{M}_{1,1}] \times  [Q_{0,2,d}(X)]\virt \times [\overline{M}_{1,1}] \, \r)  \\ \nonumber
& + \frac{1}{2} \sum_{a=0}^{\dim X-1} (-1)^a[A^{a+1}_1]_{\dim X-1-a}[B_1B_2]_{a-1} \cap \l( [\bar{M}_{1,2}]  \times [Q'_{0,2,d}(X)]\virt  \r).
\end{align}
\end{Thm*}
In Remark \ref{333} we explain $A^a_1=A^a_2$, so the last term is not so strange.

\medskip
\subsection*{Acknowledgements}
We are grateful to Jingchen Niu for delivering us his expertise on the desingularisations of the genus $2$ moduli spaces. 
We also thank Luca Battistella, Navid Nabijou, Richard Thomas for helpful comments.

\subsection*{Notation}
For a morphism $f: X \to Y$ of spaces and a perfect complex $\EE$ on $Y$, we often denote by $\EE|_X$ the derived pullback $f^*\EE$. We sometimes regard a locally free sheaf $E$ as its total space.

We denote by $\fM_{g,k,d}$, or simply by $\fM$, the Artin stack of prestable curves with non-negative integer on each component (playing a role of degree) whose sum is $d$. Similarly $\fM^{line}_{g,k,d}$, or simply $\fM^{line}$, denotes the Artin stack of curves with degree $d$ line bundles. The Artin stack of curves with degree $d$ divisors is denoted by $\fM^{div}_{g,k,d}$, or simply $\fM^{div}$. 

We denote by $Q^{(i)}$ the image of $\iota_i$ in the picture above (i) for either the moduli spaces of stable quasimaps or the $p$-fields spaces. For instance on $Q^{(3)}$, the evaluation maps (of the $g=0$ quasimap) are the same $\mathrm{ev}_1=\mathrm{ev}_2$. Furthermore, we use the script $(i)$ for relevant objects of the embedding $\iota_i$ unless it needs an explanation. For instance a bundle on $Q^{(i)}$ will be denoted with the script $(i)$.

For variables with two subindices $y_{ij}$, we say $y_i=0$ if $y_{ij}=0$ for all $j$. Also we say $y=0$ if $y_{ij}=0$ for all $i$ and $j$.

\setcounter{tocdepth}{1}
\tableofcontents

\section{Stable quasimaps, $p$-fields and the plan}\label{Sec1}

\subsection*{Stable quasimaps}
A {\em genus $g$, degree $d$ quasimap to $X$ with $k$ marked points} is a triple $(C,L,u)$ where $C$ is a genus $g$, projective, nodal, prestable curve with $k$ marked points, $L$ is a degree $d$ line bundle on $C$, and $u = (u_0,\dots,u_n)$ is a section of $L^{\oplus n+1}$ such that 
\begin{align} \label{cond}
\text{$f_i(u)=0 \in \Gamma(C, L^{\otimes \ell_i})$ for all $i$.}
\end{align}
Here $L$ plays a role of $\mathsf{f}^*\cO(1)$. It is a {\em stable quasimap} if it comes with the stability conditions\footnote{In contrast, $(C,L,u)$ is a {\em stable map} defining Gromov--Witten invariants if it is equipped with the stability conditions $1$. $\omega_{C}^{\log} \otimes L^{\otimes 3}$ is ample on $C$, and $2$. the zero of $u$ is empty.}
\begin{itemize}
\item[-]
$\omega_{C}^{\log} \otimes L^{\varepsilon}$ is ample on $C$ for any $\varepsilon >0$, and
\item[-]
the zero of $u$ is a divisor which does not meet nodes nor marked points.
\end{itemize}
We denote by $Q_{g,k,d}(X)$, or simply by $Q(X)$, the moduli space of stable quasimaps.
By \cite{MOP11, CK10, CKM14}, it is proper and equipped with a natural perfect obstruction theory so that the virtual fundamental class
\begin{align}\label{QX}
\left[ Q(X) \right]\virt \ \in\ A_{\mathrm{vdim}}(Q(X) )
\end{align}
is defined, where $\mathrm{vdim}$ denotes the virtual dimension
\[
\mathrm{vdim} \ =\ (\dim X - 3)(1-g) + k - d \cdot c_1(K_X)([line]).
\]
The {\em stable quasimap invariant of $X$} is defined to be an integration over this virtual class.

The reason why the quantum Lefschetz property for the quasimap invariants is simpler is because a quasimap does not have a rational component with less than two special points (called a rational tail) on its domain curve.

\subsection*{Stable quasimaps with $p$-fields} The coherent sheaf $R^1V:=\oplus_iR^1\pi_*\mathsf{f}^*\cO(\ell_i)$ on $Q(\PP^n)$ may not vanish. 
We denote by $Q_{p,g,k,d}(\PP^n)$, or simply by $Q_p$, its dual space 
$$
Q_p \ :=\ \Spec_{\cO_{Q(\PP^n) } } \left( \mathrm{Sym} R^1V \right).
$$
So $Q_p$ parametrises $(C,L,u,p=(p_1, ..., p_m))$ where $(C,L,u)$ is a stable quasimap to $\PP^n$ and 
\begin{align*}
p_i \in \Gamma(C, \omega_C \otimes L^{-\ell_i}).
\end{align*}
Recall that imposing the condition \eqref{cond} defines the space $Q(X)$ from $Q(\PP^n)$, whereas the above extra data determines $Q_p$ from $Q(\PP^n)$.
We will call the section $p=(p_1,\dots,p_m)$ {\em $p$-fields}.

The space $Q_p$ may not be proper, but still comes with a natural perfect obstruction theory, so that the virtual fundamental class
$$
[Q_p]\virt \ \in\ A_{\mathrm{vdim}}(Q_p)
$$
is defined.
Using the cosection 
$$
\EE\dual_{Q_p/\fM^{line}} = (R\pi_*\cL^{\oplus n+1} \oplus \bigoplus_i (R\pi_*\cL^{\otimes \ell_i})^{\vee} [-1])|_{Q_p} \longrightarrow \cO_{Q_p}[-1],
$$
defined in \cite{CL12}, cosection localisation \cite{KL13} allows us to find a localised class $[Q_p]\virt_{\mathrm{loc}}$ of $[Q_p]\virt$ to a smaller space $j: Q(X) \hookrightarrow Q_p$,
\begin{align}\label{locQp}
[Q_p]\virt_{\mathrm{loc}}\ \in\ A_{\mathrm{vdim}}(Q(X)), \ \ \ j_*[Q_p]\virt_{\mathrm{loc}} \ =\ [Q_p]\virt.
\end{align}
Then by \cite{KO18, CL20, CJW21, Pi20}, the localised class $[Q_p]\virt_{\mathrm{loc}}$ is equal to the class $[Q(X)]\virt$ for $X$ defined in \eqref{QX} up to a sign 
\begin{align}\label{X=p}
[Q(X)]\virt \ =\ (-1)^{d(\sum_i \ell_i) + m(1-g) }[Q_p]\virt_{\mathrm{loc}}.
\end{align}

\subsection*{Plan of the proof of Theorem \ref{QLP}}
We use \eqref{X=p} to prove Theorem \ref{QLP}.
A good thing to work with $Q_p$ instead of $Q(X)$ is that we can find a nice enough local cut-out model of $Q_p$ whereas it is hard for $Q(X)$.
In Section \ref{Kura}, we describe this explicit cut-out model of $Q_p$ after the suitable base-change of $Q_p$ in Section \ref{desin}.
Using this, we compute the intrinsic normal cone of $Q_p$ in Section \ref{virdecomp} to obtain a decomposition of the virtual class
\begin{align}\label{Qdec}
[Q_{p,2,0,d}]\virt_{\mathrm{loc}} \ =\ [Q_p^{\mathrm{red}}]\virt \ +\ [Q_p^{(1)}]\virt \ +\ [Q_p^{(2)}]\virt \ +\ [Q_p^{(3)}]\virt.
\end{align}
Note that the indices `red', `$(1)$', `$(2)$' and `$(3)$' reflect their geometric origins labelled above. So $Q_p^{(1)},Q_p^{(2)},Q_p^{(3)}$ are supported on the images of the node-identifying morphisms $\iota_i$, ignoring $p$-fields. In fact we will investigate that they are bundles over the images in Section \ref{virdecomp}.

\medskip
Then in Section \ref{reduced}, we prove that $[Q_p^{\mathrm{red}}]\virt$ follows the original quantum Lefschetz formula \eqref{naiveQLP1}
$$
[Q_{p}^{\mathrm{red}}]\virt \ =\ (-1)^{d(\sum_i \ell_i) - m} \; e^{\mathrm{ref}}(V_{2,0,d})\ \cap\ [Q_{2,0,d}^{\mathrm{red}}(\PP^n)].
$$ 
And we show the $i$-th cycle $[Q_p^{(i)}]\virt$ is a part of the RHS of \eqref{qlp}. For $i=1$ for instance, we obtain
\begin{align}\label{Q1vir}
[Q_p^{(1)}]\virt\ =\ (-1)^m\left[ \frac{c(\cH^\vee \boxtimes \mathrm{ev}^*T_X)}{c(\LL^\vee \boxtimes \LL^\vee)} \right]_{n-m-1} \cap  \left( [\overline{M}_{1,1}] \times  [Q^{\mathrm{red}}_{p,1,1,d}]\virt \right)
\end{align}
via the pushforward by $\iota_1$. A very brief interpretation of this equality is that the difference of the obstruction bundles defining $[Q_p^{(1)}]\virt$ and $[Q^{\mathrm{red}}_{p,1,1,d}]\virt$ (in the $K$-group of $Q_p^{(1)}$, via the pullback) can be written in terms of the bundle structure of $Q_p^{(1)}$ over the image of $\iota_1$ as well as the pullback bundles of $\cH^\vee \boxtimes \mathrm{ev}^*T_X$, $\LL^\vee \boxtimes \LL^\vee$. To realise this interpretation to give an actual proof, we do massage spaces and bundles -- deformations, blowups and twistings by divisors, etc. -- in Section \ref{LOT} so that we can get a tidy form \eqref{Q1vir}. Once we do these for all $i$, then by using \cite[Corollary 1.3]{LL22}
$$
e^{\mathrm{ref}}(V_{1,1,d}) \cap [Q_{1,1,d}^\red(\PP^n)] \ =\   [Q_{1,1,d}(X)]\virt - \frac{c_1(\LL_2)}{12} [Q_{0,2,d}(X)]\virt
$$
together with \eqref{X=p}, the decomposition \eqref{Qdec} proves Theorem \ref{QLP}.

\section{Local defining equations of the $p$-field space} \label{Kura}

For a morphism of vector bundles $d: A \to B$ over a smooth Artin stack $M$, we consider the kernel of $d$ as a space
$$
\ker d\ :=\ \Spec_{\cO_M}\left(\Sym (\mathrm{coker} d^*)\right)\ \subset \ A\ =\ \Spec_{\cO_M}\left(\Sym A^*\right).
$$
Denoting by $\tau_A$ the tautological section, $\ker d$ has a cut-out model
$$
\xymatrix@=18pt{
& B|_{A} \ar[d] \\  
\ker d \ := \ (d\circ\tau_{A})^{-1}(0)\ \subset\hspace{-6mm} & A.\ar@/^{-2ex}/[u]_{d\circ \tau_{A}}}
$$
Hence the pullback complex $\{d: A \to B\}|_{\ker d}$ defines a dual relative perfect obstruction theory of $\ker d$ over $M$.

The purpose of this section is to write $Q_p$ as (an open substack of) $\ker d$ over $\fM^{div}$.

\subsection{Cut-out model of the $p$-field space}\label{Ku}
Unlike considering $\fM^{line},$ there is no canonical forgetful morphism of the $p$-field space $Q_p \to \fM^{div}$. But it is defined locally as follows. Since $u=(u_0, ..., u_n)$ is not identically zero for a point $(C,L,u,p) \in Q_p$, we can pick a combination ${\bf u}=\sum a_iu_i\in H^0(C,L)$ nonconstant on each component of $C$ on which $L$ has positive degree. Since it is an open condition we have a morphism
$$
Q_p \ \longrightarrow\ \fM^{div}, \ \ \ (C,L,u,p) \ \longmapsto \ (C, {\bf u}^{-1}(0))
$$
on a local neighborhood.

Let $\cD$ be the universal divisor on the universal curve $\pi: \cC\to \fM^{div}$. Then locally, $Q_p$ is the open substack (defined by the stability condition) of the kernel of a representative
\begin{align} \label{localpot}
[A \stackrel{d}{\lra} B]\ \cong\ R\pi_*\cO_{\cC}(\cD)^{\oplus n} \oplus \bigoplus_i \left( R\pi_*\cO_{\cC}(\ell_i\cD)[1] \right)\dual.
\end{align} 
Hence it defines a local cut-out model relative to $\fM^{div}$ and a natural (local) perfect obstruction theory. 

Since we work locally, we may assume $A$ and $B$ are trivial bundles. Then $d$ can be considered as a multi-valued function
\begin{align}\label{OOO}
d\ :\ \fM^{div}\times\CC^{\rank A}\ \longrightarrow\ \CC^{\rank B}
\end{align}
defining $Q_p$ as (an open substack of) its zero. In the rest of the section, we find a simple expression of $d$ by coordinate changes and blowups.

\subsection{Key Lemma}
Now we focus on $(g,k)=(2,0)$ throughout the Section. We work {\em \'etale locally} on $\fM^{div}$, sometimes without mentioning it. For instance by an element of $\Gamma(\cO_{\fM^{div}})$, we mean an \'etale local function of $\fM^{div}$.

As we have explained in Introduction, considering stable quasimaps has a big advantage in making the quantum Lefschetz formula less complicated than considering stable maps. But there is (essentially only) one technical thing to check, which is obvious in stable maps -- near a domain curve of a stable map $f: C \to \PP^n$, $f^*\cO(1)$ is linearly equivalent to $\cO(\sum_{i=1}^d \cD_i)$ with disjoint, fiberwise degree $1$ divisors $\cD_i$. Unfortunately it is not immediately seen near a domain curve of a stable quasimap. Since this was the important starting point to find local cut-out models for stable map moduli spaces in \cite{HL10, HLN18} we need the following Lemma. 

In fact, the Lemma is quite general -- it holds near any prestable curve, including a domain curve of a stable quasimap, in genus $2$. Let $\cD$ be any effective divisor of $\deg =d \geq 3$ on the universal curve $\cC$ of $\fM$. 

\begin{lemm}\label{divisor splitting}
Locally $\cD$ is linearly equivalent to a sum $\sum_{i=1}^d \cD_i$ of disjoint divisors of degree $1$ at each fiber.
\end{lemm}

The key idea of the proof is to construct a covering map $\cC\to \PP^1$ by picking two linearly independent sections $H^0(\cC,\cO(\cD))$, whose dim $= d+1-g+\dim H^1(\cC,\cO(\cD))\geq 2$, not having common zeros. Then the inverse image of a generic point of $\PP^1$ is $d$-many distinct points.

\begin{proof}
Pick any local divisor $\cB$ on $\cC$ lying on the minimal genus $2$ subcurve, having degree $1$ at each fiber and not meeting $\cD$.
Because $\cB \cap \cD = \emptyset$, the evaluation morphism $\pi_*(\cO_{\cC}(\cD)) \to \cO_{\cC}(\cD)|_{\cB} \cong \cO_{\fM}$ is surjective, where $\pi : \cC \to \fM$ denotes the projection morphism. This induces an exact sequence
\begin{align}\label{DBses}
0\ \to\ \pi_* \cO_{\cC}(\cD-\cB)\ \to\ \pi_*(\cO_{\cC}(\cD))\ \to\ \cO_{\fM}\ \to\ 0. 
\end{align}

Meanwhile, as in \cite[Section 2.3]{HLN18}, we can choose other divisors $\cA_1$ and $\cA_2$ lying on the minimal genus $2$ subcurve such that
\begin{itemize}
\item
$\cA_1$, $\cA_2$, $\cB$ are disjoint to each other, and 
\item
$\cA_1$, $\cA_2$ lie on different components if the genus $2$ component consists of two genus $1$ components.
\end{itemize}
Similarly, by \cite[Equation (2.5)]{HLN18}, we obtain a sequence
\begin{align}\label{DBAses}
0\ \to\ \pi_* \cO_{\cC}(\cD -\cB)\ \to\ \xymatrix@C=20mm{\pi_* \cO_{\cC}(\cD + \cA_1 + \cA_2-\cB)\ \ar[r]^-{\mathrm{ev}_{\cA_1} \oplus \; \mathrm{ev}_{\cA_2}}& \cO_{\fM}^{\oplus 2}.}
\end{align}
Note that $\pi_* \cO_{\cC}(\cD + \cA_1 + \cA_2-\cB)$ is a rank $d$ vector bundle and hence locally is $\cO_{\fM}^{\oplus d}$. 
Since $d \geq 3$, we can pick a nonzero local section $s\in \Gamma\left(\pi_* \cO_{\cC}(\cD + \cA_1 + \cA_2-\cB)\right)$ mapping to $0$ by $\mathrm{ev}_{\cA_1} \oplus \mathrm{ev}_{\cA_2}$.
Then it factors through $\pi_* \cO_{\cC}(\cD -\cB)$, and hence, by \eqref{DBses}, it can be considered as a section 
$$
s\ :\ \cO_{\cC} \ \longrightarrow\ \cO_{\cC}(\cD), 
$$ 
zero on $\cB$. Since the canonical section $s_{\cD}$ of $\cD$ does not vanish on $\cB$, $s_{\cD}$ and $s$ are linearly independent on every fiber.

The common zero $\cD'$ of $s$ and $s_{\cD}$ has then fiberwise degree $d' \leq d-1$ (which may not be constant at each fiber) because $s$ is zero on $\cB$ but $s_{\cD}$ is not. Then at a fiber the sections $s \otimes s^{-1}_{\cD'}$, $s_{\cD} \otimes s^{-1}_{\cD'}$ of $\cO(\cD-\cD')$ defines a degree $d-d'$ morphism $\phi: \cC \to \PP^1$. Since it cannot be degree $1$ (which means $\phi$ is an isomorphism), we actually have $d' \leq d-2$.
A generic fiber $\phi^{-1}([a;b])$ consists of distinct divisors $\cD_1$, ..., $\cD_{d-d'}$ away from $\cD'$, and hence we have
$$
\cO_{\cC}(\cD-\cD') \ \cong\ \cO_{\cC}\left(\sum_{i=1}^{d-d'} \cD_i\right).
$$
Note that since $\cD' + \sum \cD_i$ is defined by $bs-as_{\cD}$ this isomorphism is not only at the fiber, but an isomorphism locally on $\fM$.

If $d' \geq 3$, we do the same procedure by replacing $\cD'$ by $\cD$ until we get $d' \leq 2$. Then we proved the lemma unless $d'=2$. Now let us assume that $d'=\deg \cD' =2$. Doing the same procedure for $\cD:=\cD'+\cD_1$ which has degree $3$, the procedure terminates since $d'\leq \deg\cD-2=3-2=1$. Hence the proof is completed.
\end{proof}

Considering the universal divisor $\cD$ on the universal curve $\cC$ on $\fM^{div}$, we obtain the following immediate corollary from the exact sequences \eqref{DBses}, \eqref{DBAses} in the proof of Lemma \ref{divisor splitting}.

\begin{coro} \label{twoisos} In the derived category of a local neighborhood of $\fM^{div}$, we obtain an isomorphism induced by \eqref{DBses}
\begin{align*}
 R\pi_* \cO_{\cC}(\cD) \ & \cong\ R\pi_* \cO_{\cC}(\cD - \cB) \oplus [\cO_{\fM^{div}} \stackrel{0}{\lra} 0]. 
\end{align*}
And the sequence \eqref{DBAses} induces an isomorphism
\begin{align*}
 R\pi_* \cO_{\cC}(\cD - \cB)  \ & \cong\ \left[ \xymatrix@C=20mm{\pi_* \cO_{\cC}(\cD + \cA_1 + \cA_2 - \cB) \ar[r]^-{\tiny{\mathrm{ev}_{\cA_1}\oplus \; \mathrm{ev}_{\cA_2} }} & \cO_{\fM^{div}}^{\oplus 2}} \right].
\end{align*}
\end{coro}

In addition, a similar idea of \cite[Lemma 2.4.1]{HLN18} allows us to have one more isomorphism.
\begin{lemm} \label{oneiso}
The canonical monomorphisms induce an isomorphism
\begin{align*}
& \oplus^d_{i=1} \pi_* \cO_{\cC}(\cD_i + \cA_1 + \cA_2 - \cB) 
\ \cong\ \pi_* \cO_{\cC}(\cD_1 + \cdots + \cD_d + \cA_1 + \cA_2 - \cB).
\end{align*}
\end{lemm}
Combining all these Lemma \ref{divisor splitting}, Corollary \ref{twoisos} and Lemma \ref{oneiso}, we observe that $R\pi_* \cO_{\cC}(\cD)$ is quasi-isomorphic to
\begin{align}\label{ess}
\left[ \xymatrix@C=15mm{\oplus^d_{i=1} \pi_* \cO_{\cC}(\cD_i + \cA_1 + \cA_2 - \cB) \ar[r]^-{\tiny{\mathrm{ev}_{\cA_1}\oplus \; \mathrm{ev}_{\cA_2} }} & \cO_{\fM^{div}}^{\oplus 2}} \right] \oplus [\cO_{\fM^{div}} \stackrel{0}{\lra} \cO_{\fM^{div}}^{\oplus 2}].
\end{align}

\subsection{Diagonalisation of the local representative}\label{diagc}
Picking any local identification $\pi_* \cO_{\cC}(\cD_i + \cA_1 + \cA_2 - \cB) \cong \cO_{\fM^{div}}$, $\tiny{\mathrm{ev}_{\cA_1}\oplus \mathrm{ev}_{\cA_2}}$ in \eqref{ess} can be written as a $2 \times d$ matrix $(c_{ji})$, $c_{ji} \in \Gamma(\cO_{\fM^{div}})$.

The goal of this section is to transform the matrix $(c_{ji})$ to a nice diagonal form
$$
(c_{ji}) \sim \left(
\begin{array}{ccccc}
c_1 & 0 & 0 & \cdots & 0 \\
0 & c_2 & 0 & \cdots & 0
\end{array}
\right) =: c
$$
by using row and column operations. In fact it is already studied by Hu-Li-Niu \cite[Section 5]{HLN18}. They found a diagonal form $c$ on a neighborhood by fixing a point in $\fM^{div}$. The description of $c$ depends on a type of a boundary component in which the point is. We list some cases which will appear as a domain curve of a stable quasimap. 

\smallskip
\noindent (1). Near a point in the image of $\overline{M}_{1,1}\times \fM^{div}_{1,1,d} \hookrightarrow \fM^{div}_{2,0,d}$ one can find a diagonal matrix $c$ to be 
$$
c_1\ =\ 1, \ \ \ c_2\ =\ \zeta,
$$
where $\zeta$ is the node smoothing function in $\Gamma(\cO_{\fM^{div}})$. The proof comes directly from \cite[Section 5.3]{HLN18}. 
\smallskip

\noindent (2). Near a point in the image of $\iota_2 : \overline{M}_{1,1}\times \fM^{div}_{0,2,d}\times \overline{M}_{1,1} \xrightarrow{2:1} \fM^{div}_{2,0,d}$ one can find it to be
$$
c_1\ =\ \zeta_1, \ \ \ c_2\ =\ \zeta_2,
$$
where $\zeta_1$ and $\zeta_2$ are the node smoothing functions in $\Gamma(\cO_{\fM^{div}})$. The proof is in \cite[Section 5.5]{HLN18}. Note that the diagonal form in (1) is recovered by $\zeta_2\neq0$. 
\smallskip

\noindent (3). Near a point in the image of $\overline{M}_{1,2}\times \fM^{div}_{0,2,d} \hookrightarrow \fM^{div}_{2,0,d}$ we need a blowup to obtain a diagonal transform of $(c_{ji})$. Before we discuss it in the following Section, we introduce some useful facts which we will use.

The entries $c_{ji}$ in the matrix $(c_{ji})$ are non-vanishing functions by \cite[Section 5.4]{HLN18}. Therefore the matrix $(c_{ji})$ can be transformed to 
\begin{align}\label{mat3}
\left(
\begin{array}{cccccc}
1 & 0 & 0 & 0 & \cdots & 0 \\
0 & \det_{12} & \det_{13} & \det_{14} & \cdots & \det_{1d}
\end{array}
\right)
\end{align}
where $\det_{k\ell} := \det \left( \begin{array}{cc} 
c_{1k} & c_{1\ell} \\ 
c_{2k} & c_{2\ell} 
\end{array}
\right)$.
By \cite[Section 5.4]{HLN18} and \cite[Lemma 2.8.2]{HLN18}, we may assume that the first two determinants can be written as
\[
\det_{12} = \zeta_1 + a \cdot \zeta_2,\ \ \ \det_{13} = \zeta_2 + b \cdot \zeta_1,
\]
where $\zeta_1$ and $\zeta_2$ are the node smoothing functions.
\smallskip

\noindent (4). Near a generic domain curve from the reduced space, one can find it to be
$$
c_1\ =\ 1, \ \ \ c_2\ =\ 1.
$$
The proof is in \cite[Section 5.2]{HLN18}. 
This diagonal form is recovered from (1) by letting $\zeta\neq0$.

\subsection{Base change}\label{desin}
Consider the blowup spaces
$$
\wtil{\fM} := \Bl_{\fM_{1,2,0} \times \fM_{0,2,d}} \fM_{2,0,d} \ \textrm{ and } \ 
\widetilde{\fM}^{div} \ :=\ \fM^{div}\times_{\fM} \wtil{\fM}.
$$
On $\widetilde{\fM}^{div}$, the matrix \eqref{mat3} can be transformed to be a diagonal form.

Locally the boundary component $\fM_{1,2,0} \times \fM_{0,2,d}$ is $\{\zeta_1=\zeta_2=0\}$. Thus on a neighborhood of the exceptional divisor, we know either $\zeta_1|\zeta_2$ or $\zeta_2|\zeta_1$. Without loss of generality, we may assume that $\zeta_1|\zeta_2$. Then the matrix \eqref{mat3} can be transformed to 
$$
\left(
\begin{array}{ccccc}
1 & 0 & 0 & \cdots & 0 \\
0 & \zeta_1 & 0 & \cdots & 0
\end{array}
\right).
$$
Hence on the blowup, the matrix $c$ for the case (3) in Section \ref{diagc} has a form with $c_1=1$, $c_2=\zeta_1$. Furthermore, the diagonal form in (4) is recovered by this by $\zeta_1\neq 1$.

The global forgetful morphism $Q_p \to \fM$ defines the base change $b:\widetilde{Q}_p := Q_p \times_{\fM} \wtil{\fM}\to Q_p$. Then the pullbacks give rise to the cut-out model, perfect obstruction theory, and cosection so that the cosection localised virtual cycle
$$
[\widetilde{Q}_p]\virt _{\mathrm{loc}} \ \in\ A_{\mathrm{vdim}}(\widetilde{Q}_X)
$$ 
is defined. By \cite[Theorem 5.0.1]{Co06}, we have
$$
b_*[\widetilde{Q}_p]\virt_{\mathrm{loc}} \ =\ [Q_p]\virt_{\mathrm{loc}}.
$$

\subsection{Local cut-out model of $\widetilde{Q}_p$}\label{KuQt}
Recall that we obtained an explicit representative \eqref{ess} of $R\pi_*\cO(\cD)$ with the diagonal matrices $c$ in Section \ref{diagc} and Section \ref{desin} as its differential morphism. We emphasise once again that $\cD$ need not be the universal divisor, cf. Lemma \ref{divisor splitting}. So we apply these diagonalisations to get a local cut-out model not only of $Q(\PP^n)$, but also of the $p$-field space $\widetilde{Q}_p$, relative over $\wtil{\fM}^{div}$ as discussed in Section \ref{Ku}. The induced local defining equation \eqref{OOO} is
\begin{align}\label{KuModel}
\xymatrix@C=0mm@R=7mm{
\CC^{2n} \times \prod^m_{i=1} \left( \CC^2 \oplus \CC^{d\ell_i-1} \right) &\ni & (c_1(z)x_{1j}, c_2(z)x_{2j}) \times \prod_i \left( (c_1(z) p_{1i}, c_2(z)p_{2i}) , 0 \right) \\
\widetilde{\fM}^{div} \times \prod_{j=1}^{n} \left(\CC^{2} \times \CC^{d-1}\right) \times \CC^{2m} \ar[u]_{ c\; \circ \tau} & \ni & \{z\} \times ((x_{1j},x_{2j}))_{1\leq j \leq n} \times \{v\} \times ((p_{1i},p_{2i}))_{1 \leq i \leq m}. \ar@{|->}[u]_{ c\; \circ \tau}
}
\end{align}
We need some explanation.
The morphism $\prod_{j=1}^n\left(\CC^{2} \times \CC^{d-1}\right) \to \CC^{2n}$ above is represented by $R\pi_*\cO(\cD)^{\oplus n}$ and $\CC^{2m}\to \prod^m_{i=1} \left( \CC^2 \oplus \CC^{d\ell_i-1} \right)$ is represented by $\left(\oplus_{i}R\pi_*\cO(\ell_i\cdot\cD)[1]\right)^\vee$, whose direct sum defines a local perfect obstruction theory \eqref{localpot}.

\section{Perfect obstruction theories, cones and virtual cycles}

\subsection{Perfect obstruction theories}\label{LtoG}
Although the cut-out model \eqref{KuModel} is useful in computational aspects, there are also two crucial drawbacks. One is it is not global and the other is this does not give a cut-out model over $\wtil{\fM}^{line}$ since $\wtil{\fM}^{div}\to\wtil{\fM}^{line}$ is not smooth. For later use it is important how we can apply computations with the cut-out model \eqref{KuModel} to the perfect obstruction theory over $\wtil{\fM}$ or $\wtil{\fM}^{line}$. In this section, we explain this.

\smallskip
First we recall the perfect obstruction theories. The {\em local} relative perfect obstruction theory over $\wtil{\fM}^{div}$ is
\begin{align*}
\EE_{\wtil{Q}_p/ \wtil{\fM}^{div}} = \left(R\pi_*\cO_{\cC}(\cD)^{\oplus n} \right)^\vee\oplus \bigoplus_i R\pi_*\cO_{\cC}(\ell_i\cD) [1],
\end{align*}
which is just the pullback of \eqref{localpot}. Over $\wtil{\fM}^{line}$, $\wtil{Q}_p$ is equipped with the {\em global} perfect obstruction theory
$$
\EE_{\wtil{Q}_p/ \wtil{\fM}^{line}} = \left(R\pi_*\cL^{\oplus n+1} \right)^\vee\oplus \bigoplus_i R\pi_*\cL^{\ell_i} [1],
$$
where $\cL$ is the universal line bundle over the universal curve $\pi : \cC \to \wtil{Q}_p$. And over $\wtil{\fM}$ the cone of the composition
$$
\EE_{\wtil{Q}_p/ \wtil{\fM}^{line}}[-1]\ \longrightarrow\ \LL_{\wtil{Q}_p/ \wtil{\fM}^{line}}[-1]\ \longrightarrow\ \LL_{\wtil{\fM}^{line}/\wtil{\fM}}|_{\wtil{Q}_p}
$$
defines the {\em global} perfect obstruction theory $\EE_{\wtil{Q}_p/ \wtil{\fM}}$. Here $\LL$ denotes the cotangent complex. Then we have the following diagram of triangles
$$
\xymatrix@C=4mm@R=4mm{
\EE_{\wtil{Q}_p/ \wtil{\fM}} \ar[r]\ar@{=}[d] & \EE_{\wtil{Q}_p/ \wtil{\fM}^{line}}\ar[r]\ar[d] & \LL_{\wtil{\fM}^{line}/\wtil{\fM}}|_{\wtil{Q}_p}[1] \ar[d] \\
\EE_{\wtil{Q}_p/ \wtil{\fM}} \ar[r] & \EE_{\wtil{Q}_p/ \wtil{\fM}^{div}}\ar[r]\ar[d] & \LL_{\wtil{\fM}^{div}/\wtil{\fM}}|_{\wtil{Q}_p}[1]\ar[d]\\
&\LL_{\wtil{\fM}^{div}/\wtil{\fM}^{line}}|_{\wtil{Q}_p}[1] \ar@{=}[r] & \LL_{\wtil{\fM}^{div}/\wtil{\fM}^{line}}|_{\wtil{Q}_p}[1].
}
$$
In particular the middle horizontal triangle tells us that the local cut-out model \eqref{KuModel} defines $\EE_{\wtil{Q}_p/ \wtil{\fM}}$ as well as $\EE_{\wtil{Q}_p/ \wtil{\fM}^{div}}$ since $\wtil{\fM}^{div}\to\wtil{\fM}$ is smooth.\footnote{Beware that the local model \eqref{KuModel} does not define $\EE_{\wtil{Q}_p/\wtil{\fM}^{line}}$ immediately because $\wtil{\fM}^{div}\to \wtil{\fM}^{line}$ is not smooth.}

So one way from local to global is to consider this forgetful morphism $\wtil{\fM}^{div} \to \wtil{\fM}$. Via the morphism of perfect obstruction theories
$$
\EE_{\wtil{Q}_p/\wtil{\fM}}\ \longrightarrow\ \EE_{\wtil{Q}_p/\wtil{\fM}^{div}},
$$
computations can move from one to the other, where the former is global whereas the latter is local. For instance the smoothness shows that the two intrinsic normal cones 
\begin{align*}
\fC_{\wtil{Q}_p / \wtil{\fM}^{div}} \ \text{ and }\ \fC_{\wtil{Q}_p / \wtil{\fM}}
\end{align*} 
are related, the former maps to the latter via the morphism of bundle stacks
\begin{align*}
h^1/h^0\left(\EE_{\wtil{Q}_p/\wtil{\fM}^{div}}\dual\right)\ \longrightarrow\ h^1/h^0\left(\EE_{\wtil{Q}_p/\wtil{\fM}}\dual\right),
\end{align*}
which is actually an affine $T_{\wtil{\fM}^{div}/\wtil{\fM}}$-bundle. The precise proof is in \cite[Proposition 3]{KKP03}, but it is more or less obvious thanks to the smoothness. Then the local computation of the cone on the LHS using the cut-out model \eqref{KuModel} will give the computation of the cone on the RHS.

\medskip

A solution to $\wtil{\fM}^{line}$ is to consider the forgetful morphism $\wtil{\fM}^{line} \to \wtil{\fM}$. Since it is smooth as well the morphism of perfect obstruction theories
$$
\EE_{\wtil{Q}_p/\wtil{\fM}}\ \longrightarrow\ \EE_{\wtil{Q}_p/\wtil{\fM}^{line}},
$$
induces the relationship of the two intrinsic normal cones
\begin{align*}
\fC_{ \wtil{Q}_p / \wtil{\fM}^{line}}\ \text{ and }\ \fC_{\wtil{Q}_p / \wtil{\fM}},
\end{align*} 
namely, the former maps to the latter via the morphism of bundle stacks
\begin{align*}
h^1/h^0\left(\EE_{\wtil{Q}_p/\wtil{\fM}^{line}}\dual\right)\ \longrightarrow\ h^1/h^0\left(\EE_{\wtil{Q}_p/\wtil{\fM}}\dual\right)
\end{align*}
as before. It is also an affine $T_{\wtil{\fM}^{line}/\wtil{\fM}}$-bundle, and so is $\fC_{ \wtil{Q}_p / \wtil{\fM}^{line}}$ over $\fC_{\wtil{Q}_p / \wtil{\fM}}$.

\subsection{Virtual cycles} \label{virdecomp}
As it is briefly explained in Section \ref{Sec1}, the space $\widetilde{Q}_p$ is decomposed into four irreducible components
\begin{align}\label{qdecomp}
\widetilde{Q}_p \ =\ \wtil{Q}_p^{\mathrm{red}} \ \cup\ \wtil{Q}_p^{(1)}  \ \cup\ \wtil{Q}_p^{(2)}  \ \cup\ \wtil{Q}_p^{(3)},
\end{align}
cf. the pictures in Introduction. From the local cut-out model \eqref{KuModel} relative over $\wtil{\fM}^{div}$, an \'etale local neighborhood of $\widetilde{Q}_p$ is the spectrum of a ring
$$
R\ :=\  B[x,p]\; /(c_1x_{1j}, c_2x_{2j},c_1p_{1i}, c_2p_{2i}),
$$
where $\Spec(B)$ is a smooth neighborhood of $\wtil{\fM}^{div}$. From this we can read the decomposition \eqref{qdecomp} as follows:

\smallskip
\noindent
(1). Near a point in $\widetilde{Q}_p$ whose domain curve is an element in the image of $\overline{M}_{1,1}\times \fM_{1,1} \hookrightarrow \fM_{2,0}$ but not in the case (2) below, we have seen $c_2=1$ in Section \ref{diagc}. Hence
$$
\wtil{Q}_p^{(1)}= \{c_1=x_2=p_2=0\}, \ \ \wtil{Q}_p^{\mathrm{red}} = \{x=p=0\}.
$$
\smallskip
\noindent
(2). Near that in the image of $\overline{M}_{1,1}\times \fM_{0,2}\times \overline{M}_{1,1} \xrightarrow{2:1} \fM_{2,0}$, we have
\begin{align*}
&\wtil{Q}_p^{(2)} = \{c_1=c_2=0\}, \ \ \wtil{Q}_p^{\mathrm{red}} = \{x=p=0\}, \\
\text{and a double }&\text{cover }\ \{c_1=x_2=p_2=0\}\cup \{c_2=x_1=p_1=0\} \ \longrightarrow\ \wtil{Q}_p^{(1)}.
\end{align*}
Note that $\wtil{Q}_p^{(2)}$ does not meet $\wtil{Q}_p^{(3)}$.

\smallskip
\noindent
(3). Near that in the image of $\overline{M}_{1,2}\times \fM_{0,2} \hookrightarrow \fM_{2,0}$ but not in the case (2) above, we have $c_2=1$. So
$$
\wtil{Q}_p^{(3)}= \{c_1=x_2=p_2=0\}, \ \ \wtil{Q}_p^{\mathrm{red}} = \{x=p=0\}.
$$
\smallskip
\noindent
(4). Near a point outside there, we have $c_1=c_2=1$. Thus $\Spec(R)$ defines $\wtil{Q}_p^{\mathrm{red}}$.


Then the intrinsic normal cone $\fC_{\widetilde{Q}_p / \wtil{\fM}}$ can be decomposed into
\begin{align}\label{others}
\fC_{\widetilde{Q}_p / \wtil{\fM}} \ =\ \fC^{\mathrm{red}}\ \cup\ \fC^{(1)}\ \cup\ \fC^{(2)}\ \cup\ \fC^{(3)} \ \cup\ \text{others},
\end{align}
each of the first four terms is defined to be the closure of the complement open part in $\fC_{\widetilde{Q}_p / \wtil{\fM}}$. For instance, 
$$
\fC^{\mathrm{red}}\ \text{ is the closure of }\ \fC_{\widetilde{Q}_p / \wtil{\fM}} |_{\widetilde{Q}_p \setminus \wtil{Q}_p^{(1)} \cup \wtil{Q}_p^{(2)} \cup \wtil{Q}_p^{(3)}}\ \subset\ \fC_{\widetilde{Q}_p / \wtil{\fM}}.
$$ 
They are actually the closures in $h^1/h^0\left(\EE_{\wtil{Q}_p/\wtil{\fM}}\dual\right)$ since $\fC_{\widetilde{Q}_p / \wtil{\fM}}\subset h^1/h^0\left(\EE_{\wtil{Q}_p/\wtil{\fM}}\dual\right)$ is a closed substack.
In fact, one can check from the cut-out model \eqref{KuModel} that `others' in \eqref{others} is empty so that we obtain a decomposition
\begin{align}\label{Cdec}
\fC_{\widetilde{Q}_p / \wtil{\fM}} \ =\ \fC^{\mathrm{red}}\ \cup\ \fC^{(1)}\ \cup\ \fC^{(2)}\ \cup\ \fC^{(3)}.
\end{align}
Here is a brief explanation. Letting $A:=B[x,p]$, one can read the decomposition of $C_{R/A}:=C_{\Spec R/\Spec A}$, a pullback of $\fC_{\widetilde{Q}_p / \wtil{\fM}^{div}}$, from its spectrum of
\begin{align}\label{CRS}
\frac{R\; [X_{1j},X_{2j},P_{1i},P_{2i}]}{\left(
\begin{array}{c}
x_{1k}X_{1l}-x_{1l}X_{1k},\ x_{1k}P_{1l}-p_{1l}X_{1k},\ p_{1k}P_{1l}-p_{1l}P_{1k}, \\
x_{2k}X_{2l}-x_{2l}X_{2k},\ x_{2k}P_{2l}-p_{2l}X_{2k},\ p_{2k}P_{2l}-p_{2l}P_{2k}
\end{array}
\right).}
\end{align}
\smallskip
\noindent
(1). Near a point over $\wtil{Q}^{(1)}_p$ ($c_2=1$, $x_2=p_2=0$), $C_{R/A}$ is decomposed into
$$
C^{\;\!(1)}= \{c_1=x_2=p_2=0\}, \ \ C^{\;\!\mathrm{red}} = \{x=p=0\}.
$$
\smallskip
\noindent
(2). Near a point over $\wtil{Q}^{(2)}_p$, $C_{R/A}$ is decomposed into
\begin{align*}
&C^{\;\!(2)} = \{c_1=c_2=0\}, \ \ C^{\;\!(1)} = \{c_1=x_2=p_2=0\}\cup \{c_2=x_1=p_1=0\}, \ \ C^{\;\!\mathrm{red}} = \{x=p=0\}.
\end{align*}
\smallskip
\noindent
(3). Near a point over $\wtil{Q}^{(3)}_p$ ($c_2=1$, $x_2=p_2=0$), $C_{R/A}$ is decomposed into
$$
C^{\;\!(3)}= \{c_1=x_2=p_2=0\}, \ \ C^{\;\!\mathrm{red}} = \{x=p=0\}.
$$
\smallskip
\noindent
(4). Near a point over $\wtil{Q}^{\mathrm{red}}_p$ ($c_2=1$, $x_2=p_2=0$), $C_{R/A}$ is $C^{\;\!\mathrm{red}} = \{x=p=0\}$.

\smallskip
\noindent
So we could check there is no `others' in $\fC_{\widetilde{Q}_p / \wtil{\fM}^{div}}$. Combining this with the (local) equivalence of
$$
\fC_{ \wtil{Q}_p / \wtil{\fM} }\ \text{ and }\ \fC_{\wtil{Q}_p / \wtil{\fM}^{div}}  \ = \   [C_{R/A} /T_{A/B}|_R ]
$$
discussed in Section \ref{LtoG} gives the decomposition \eqref{Cdec}.

Note that the cut-out model \eqref{KuModel} tells us the morphism $d(c \circ \tau):T_A|_R \to C_{R/A}$ (defining the quotient via the composition $T_{A/B}\to T_A$) is
\begin{align} \label{normal}
\partial_{x_1}, \partial_{x_2}, \partial_{p_1}, \partial_{p_2} \ & \longmapsto \ c_1\partial_{X_1}, c_2\partial_{X_2}, c_1\partial_{P_1}, c_2\partial_{P_2}, \nonumber \\ 
\partial_{c_1}, \partial_{c_2} \ & \longmapsto \ \sum_jx_{1j}\partial_{X_{1j}} + \sum_ip_{1i}\partial_{P_{1i}}, \sum_jx_{2j}\partial_{X_{2j}} + \sum_ip_{2i}\partial_{P_{2i}}.
\end{align}

The cosection introduced in \cite{CL12} defining the localised virtual cycle $[Q_p]\virt_{\loc}$ mentioned in \eqref{locQp} is indeed defined on the obstruction sheaf $h^1\left(\EE^\vee_{Q_p/\fM^{line}}\right)$ {\em over} $\fM^{line}$. So this gives a morphism $\EE^\vee_{Q_p/\fM^{line}}\to \cO_{Q_p}[-1]$ in the derived category. It is proven in \cite{CL12} that this actually factors through the absolute dual perfect obstruction theory $\EE^\vee_{Q_p}\to \cO_{Q_p}[-1]$. So its pullback defines cosections on both obstruction sheaves over $\wtil{\fM}$ and $\wtil{\fM}^{div}$, $h^1\left(\EE_{\wtil{Q}_p/\wtil{\fM}}\right)$ and $h^1\left(\EE_{\wtil{Q}_p/\wtil{\fM}^{div}}\right)$. The latter is
\begin{align}\label{COsect}
\CC^{2n}\times\CC^{2m} \ &\longrightarrow\ \CC, \\
(X_{1j}, X_{2j}, P_{1i}, P_{2i}) \ & \longmapsto \ \sum_{i,j,k}\left(p_{ki}\frac{df_i}{x_{kj}}(x_k)\cdot X_{kj} -\deg f_i \cdot f_i(x_k)\cdot P_{ki}\right)\nonumber
\end{align}
in the cut-out model \eqref{KuModel}. Here we used the restriction $\CC^{2n}\times\CC^{2m}\subset \CC^{2n} \times \prod^m_{i=1} \left( \CC^2 \oplus \CC^{d\ell_i-1} \right)$ of the obstruction bundle. It is easy to check that the composition with $c \circ \tau$ is zero. Since the cut-out model \eqref{KuModel} defines actually the absolute perfect obstruction theory, it gives another simple proof that the cosection descends to the obstruction sheaf of the absolute perfect obstruction theory. 

\medskip
For Definition below, we use the perfect obstruction theory {\em over} $\wtil{\fM}$. Especially we use the decomposition \ref{Cdec}.
\begin{defi}\label{CYCLES}
The virtual cycle of the reduced part
$$
[\wtil{Q}^{\mathrm{red}}_p]\virt \ \in \ A_{\mathrm{vdim}} (\widetilde{Q}_X \cap \wtil{Q}^{\mathrm{red}}_p)
$$ 
is defined by the image of $[\fC^{\mathrm{red}}]$ by the cosection localised Gysin map \cite{KL13}. The cycles $[\wtil{Q}^{(1)}_p]\virt$, $[\wtil{Q}^{(2)}_p]\virt$ and $[\wtil{Q}^{(3)}_p]\virt$ are similarly defined by using $[\fC^{(1)}]$, $[\fC^{(2)}]$ and $[\fC^{(3)}]$ respectively.
\end{defi}

Hence we obtain a decomposition of the virtual class
\begin{align*}
[\widetilde{Q}_{p}]\virt_{\mathrm{loc}} \ =\ [\wtil{Q}_p^{\mathrm{red}}]\virt \ +\ [\wtil{Q}_p^{(1)}]\virt \ +\ [\wtil{Q}_p^{(2)}]\virt \ +\ [\wtil{Q}_p^{(3)}]\virt
\end{align*}
providing \eqref{Qdec} by the pushdown.

\section{Quantum Lefschetz property for the reduced virtual cycle} \label{reduced}
As we have explained in Section \ref{LtoG}, the cone $\fC_{\wtil{Q}_p / \wtil{\fM}^{line}}$ is an affine bundle over $\fC_{\wtil{Q}_p / \wtil{\fM}}$. Hence the decomposition \eqref{Cdec} of $\fC_{\wtil{Q}_p / \wtil{\fM}}$ provides a decomposition of $\fC_{\widetilde{Q}_p / \wtil{\fM}^{line}}$, by abuse of notation,
\begin{align}\label{Cdec2}
\fC_{\widetilde{Q}_p / \wtil{\fM}^{line}} \ =\ \fC^{\mathrm{red}}\ \cup\ \fC^{(1)}\ \cup\ \fC^{(2)}\ \cup\ \fC^{(3)}.
\end{align}
The functorial property of localised Gysin homomorphisms \cite{KL13} tells us that the cycle $[\wtil{Q}_p^\red]\virt$ defined by using the subcone of $\fC_{\wtil{Q}_p / \wtil{\fM}}$ is the same as the one defined by using $\fC^{\mathrm{red}}$, a subcone of $\fC_{\widetilde{Q}_p / \wtil{\fM}^{line}}$ in \eqref{Cdec2}. We will compute the latter one in this section. 

The bundle $V:=V_{2,0,d}$ in Theorems \ref{QLP1} and \ref{QLP} is precisely defined to be
$$
V\ :=\ \oplus_i \pi_*\cL^{\otimes \ell_i}\ =\ h^0\left(\EE_{\wtil{Q}_p/\wtil{Q}(\PP^n)}[-1]|_{\wtil{Q}_p^{\mathrm{red}}}\right),
$$
where $\cL$ is the universal line bundle on the universal curve $\pi:\cC\to \wtil{Q}^{\red}_p$, and obviously $\wtil{Q}(\PP^n)$ is the base change $Q_{2,0,d}(\PP^n)\times_{\fM}\wtil{\fM}$.
Each point in $\wtil{Q}_p^{\mathrm{red}}$ has a section data $u \in \Gamma(C, L^{\oplus n+1})$ via the morphism $\wtil{Q}_p^{\mathrm{red}} \hra \wtil{Q}_p$. Then the morphism
$$
\EE^{\vee}_{\wtil{Q}(\PP^n)/\wtil{\fM}^{line}}=R\pi_*\cL^{\oplus n+1}\ \longrightarrow\ \EE_{\wtil{Q}_p/\wtil{Q}(\PP^n)}[-1]=\oplus_i R\pi_*\cL^{\otimes \ell_i}
$$
induced by $f_1, ..., f_m$ takes the universal section $u$ to a section $(f_i(u))_{1\leq i \leq m} \in \Gamma\left(V\right)$ in $h^0$. It defines the refined Euler class $e^{\mathrm{ref}}(V)$ for Theorems \ref{QLP1} and \ref{QLP}. 

Note that $(-1)^{\mathrm{rank}V}=(-1)^{d(\sum_i \ell_i) - m}$ and $\wtil{Q}^{\mathrm{red}}_p = \wtil{Q}^{\mathrm{red}}(\PP^n)$, which is smooth.
\begin{prop}\label{prop4.1}The reduced virtual cycle satisfies the original quantum Lefschetz formula \eqref{naiveQLP1}
$$
[\wtil{Q}^{\mathrm{red}}_p]\virt\ =\ (-1)^{d(\sum_i \ell_i) - m}e^{\mathrm{ref}}(V) \ \cap\ [\wtil{Q}^{\mathrm{red}}(\PP^n)].
$$
\end{prop}
\begin{proof}
The cut-out morphism \eqref{KuModel} restricted to $\wtil{Q}^{\mathrm{red}}_{p}\times \CC^{2m}$ mapping to $\prod^m_{i=1} \left( \CC^2 \oplus \CC^{d\ell_i-1} \right)$ gives rise to a cut-out model defining the perfect obstruction theory $\EE_{\wtil{Q}_p/\wtil{Q}(\PP^n)}|_{\wtil{Q}_p^{\mathrm{red}}}$. This morphism is precisely 
\begin{align*}
g\ :\ \wtil{Q}^{\mathrm{red}}_{p}\times \CC^{2m}\ \longrightarrow\ \prod^m_{i=1} \left( \CC^2 \oplus \CC^{d\ell_i-1} \right), \ \ u\times (p_{1i}, p_{2i})\ \longmapsto\ (c_1p_{1i}, c_2p_{2i}),
\end{align*}
and the dual $\left(dg|_{\wtil{Q}^{\mathrm{red}}_p}\right)^\vee$ defines $\EE_{\wtil{Q}_p/\wtil{Q}(\PP^n)}|_{\wtil{Q}_p^{\mathrm{red}}}$. So $V$ is the dual torsion-free part of the cokernel of $dg|_{\wtil{Q}^{\mathrm{red}}_p}$, which is the same as the cokernel of $dg'|_{\wtil{Q}^{\mathrm{red}}_p}$, where
\begin{align*}
g'\ :\ \wtil{Q}^{\mathrm{red}}_{p}\times \CC^{2m}\ \longrightarrow\ \prod^m_{i=1} \left( \CC^2 \oplus \CC^{d\ell_i-1} \right), \ \ u\times (p_{1i}, p_{2i})\ \longmapsto\ (p_{1i}, p_{2i}).
\end{align*}
Clearly the cut-out model $g'$ defines $(-1)^{d(\sum_i \ell_i) - m}e^{\mathrm{ref}}(V) \cap [\wtil{Q}^{\mathrm{red}}_{p}]$. Note that the dual of the cosection \eqref{COsect} is $(-\deg f_i\cdot f_i(u))_{1\leq i \leq m} \in \Gamma\left(V\right)$ which deforms to the defining section $(f_i(u))_{1\leq i \leq m} \in \Gamma\left(V\right)$ without changing the zero locus, defining $e^{\mathrm{ref}}(V)$. On the other hand, the normal cone defined by using the model $g'$ gives $\fC^{\mathrm{red}}$ in \eqref{Cdec2} through the computation using \eqref{CRS}. Hence it also defines $[\wtil{Q}^{\mathrm{red}}_p]\virt$.
\end{proof}

\section{Lower genus contributions from the rest cycles}\label{LOT}

\subsection{Cones in the obstruction bundle}\label{sect:conesinobs}
In this section we consider our space $\wtil{Q}_p$ {\em over $\wtil{\fM}$}, so use the perfect obstruction theory $\EE_{\wtil{Q}_p /\wtil{\fM}}$, the decomposition \eqref{Cdec} and Definition \ref{CYCLES} for virtual cycles. Letting
$$
A\ :=\ \widetilde{\fM}^{div} \times \prod_{j=1}^{n} \left(\CC^{2} \times \CC^{d-1}\right) \times \CC^{2m}
$$
be the local smooth space in the cut-out model \eqref{KuModel} having forgetful map $A\to \wtil{\fM}$, the dual perfect obstruction theory $\EE^{\vee}_{\wtil{Q}_p /\wtil{\fM}}$ is locally isomorphic to
$$
\left. \left[\, T_{A/\widetilde{\fM}} \ \xrightarrow{d(c \circ \tau)} \ \cO_A^{\oplus 2n} \oplus \bigoplus_i \cO_A^{\oplus d\ell_i +1} \, \right]  \right|_{\widetilde{Q}_p}. 
$$
Using this local expression we check that $h^{-1} \left(\EE_{\wtil{Q}_p / \wtil{\fM}}|_{\wtil{Q}_p^{(i)}} \right)$ is locally free. We denote its dual by $E^{(i)}$\footnote{This is not $h^{1}\left(\EE^\vee_{\wtil{Q}_p / \wtil{\fM}}|_{\wtil{Q}_p^{(i)}}\right)$, since it may not be locally free.}.

Picking any global locally free representative $[F_0 \stackrel{d}{\lra} F_1]$ of $\EE^\vee_{\wtil{Q}_p/\wtil{\fM}}$, we obtain a diagram
$$
\xymatrix@R=5mm{
& F_1 |_{\widetilde{Q}_p^{(i)}} \ar[d] \ar[r] & E^{(i)} \\
\fC^{(i)}\ \ar@{^(->}[r] &\ [F_1/F_0]|_{\widetilde{Q}_p^{(i)}} .
}
$$
Using this we define $C^{(i)}\hookrightarrow E^{(i)}$ to be the image of the pullback of $\fC^{(i)}\hookrightarrow [F_1/F_0]$. Then its image by the cosection localised Gysin map is $[\wtil{Q}_p^{(i)}]\virt$ by Definition \ref{CYCLES}. We denote this Kiem-Li's cosection localised Gysin map by $e^{\KL}(E^{(i)})$\footnote{This notation seems not too strange because it follows the properties of Euler classes since it is a bivariant class in rational coefficients \cite{KO18}.} so that
\begin{align}\label{a}
[\wtil{Q}_p^{(i)}]\virt \ =\ e^{\KL}(E^{(i)})\cap [C^{(i)}].
\end{align}

\medskip

\noindent Now, let us consider other intrinsic normal cones $\fC_{\widetilde{Q}_p^{(i)}/\wtil{\fM}^{(i)}}$, where $\wtil{\fM}^{(i)}\subset\wtil{\fM}$ is the image of
\begin{enumerate}
\item $\wtil{\fM}_{1,1,0} \times \wtil{\fM}_{1,1,d}$,
\item $\wtil{\fM}_{1,1,0} \times \wtil{\fM}_{0,2,d} \times \wtil{\fM}_{1,1,0}$,
\item $\wtil{\fM}_{1,2,0} \times \wtil{\fM}_{0,2,d}$,
\end{enumerate}
under the node-identifying morphism. These cones are the bundle stacks, zero sections of $h^1/h^0$ of the tangent complexes $\mathbb{T}_{\widetilde{Q}_p^{(i)}/\wtil{\fM} }$ because $\widetilde{Q}_p^{(i)} \to \wtil{\fM}^{(i)}$ is smooth. Meanwhile the morphism $\mathbb{T}_{\widetilde{Q}_p^{(i)}/\wtil{\fM}} \to \mathbb{T}_{\widetilde{Q}_p/\wtil{\fM}}|_{\widetilde{Q}_p^{(i)}}\to \EE^\vee_{\widetilde{Q}_p/\wtil{\fM}}|_{\widetilde{Q}_p^{(i)}}$ induces a representable morphism of bundle stacks
$$
h^1/h^0\left(\mathbb{T}_{\widetilde{Q}_p^{(i)}/\wtil{\fM}}\right)\ \longrightarrow\ [F_1/F_0]|_{\widetilde{Q}_p^{(i)}}.
$$
Then defining the cone
$$
C_{(i)}\ :=\ N_{\wtil{\fM}^{(i)}/\wtil{\fM}}|_{\wtil{Q}_p^{(i)}}, 
$$
its base change takes $C_{(i)}$ to
\begin{align}\label{eq:coneEmbed}
C_{(i)}=\ \fC_{\widetilde{Q}_p^{(i)}/\wtil{\fM}^{(i)}}\times_{\fC_{\widetilde{Q}_p^{(i)}/\wtil{\fM}^{(i)}}}N_{\wtil{\fM}^{(i)}/\wtil{\fM}}|_{\wtil{Q}_p^{(i)}}\ \longrightarrow\ \fC_{\widetilde{Q}_p^{(i)}/\wtil{\fM}^{(i)}}\times_{[F_1/F_0]} F_1\ \longrightarrow\ F_1|_{\widetilde{Q}_p^{(i)}}\ \longrightarrow\ E^{(i)}
\end{align}
through the above bundle stack morphism. Using \eqref{normal}, we see that the first arrow is locally 
\begin{enumerate}
\item $\partial_{c_1} \in C_{(1)}  \longmapsto d(c \circ \tau)(\partial_{c_1})$,
\item $\partial_{c_1}, \partial_{c_2} \in C_{(2)} \longmapsto d(c \circ \tau)(\partial_{c_1}), d(c \circ \tau)(\partial_{c_2})$,
\item $\partial_{c_1} \in C_{(3)} \longmapsto d(c \circ \tau)(\partial_{c_1})$.
\end{enumerate}
Since $d(c \circ \tau)(\partial_{c_1})$ annihilates the defining equations of $C^{(i)}$, for instance
$$
d(c \circ \tau)(\partial_{c_1}) (d(x_{1l}X_{1k}-x_{1k}X_{1l}))\ =\ 0,
$$
the morphism $C_{(i)} \to E^{(i)}$ factors through
$$
C_{(i)}\ \longrightarrow\ C^{(i)}\ \hookrightarrow\ E^{(i)}.
$$
In fact $C_{(i)}$ maps isomorphic to $C^{(i)}\subset E^{(i)}$ on which $d(c \circ \tau)(\partial_{c_j})$ does not vanish. Note that it vanishes on either $x_1=p_1=0$ or $x_2=p_2=0$. Since $C_{(i)}$ is a bundle we may expect an advantage of using $C_{(i)}$ instead of $C^{(i)}$ for \eqref{a} if this is possible in a certain way. But this is not an absurd fantasy since they are almost isomorphic.

\begin{exam}\label{EXAM1}
The local structure ring \eqref{CRS} tells us that $C^{(1)}$ is (locally) the spectrum of 
\begin{align*}
\frac{R/(x_2,p_2,c_1) \; [X_{1j},P_{1i}]}{\left(
\begin{array}{c}
x_{1k}X_{1l}-x_{1l}X_{1k},\ x_{1k}P_{1l}-p_{1l}X_{1k},\ p_{1k}P_{1l}-p_{1l}P_{1k}
\end{array}
\right).}
\end{align*}
Meanwhile by its definition $C_{(1)}$ is (locally) the spectrum of $R/(x_2,p_2,c_1)[Y]$, where the variable $Y$ is a coordinate of $\partial_{c_1}$. So the morphism $C_{(1)} \to C^{(1)}$ is
$$
X_{1j} \ \longmapsto\ x_{1j}Y, \ \ \ P_{1i} \ \longmapsto\ p_{1i}Y.
$$
\end{exam}

\subsection{Outline of the proof of Theorem \ref{QLP}}\label{OUTLINE}
Letting $p_1=p_2=0$, $\eqref{KuModel}$ gives a local cut-out model of $\wtil{Q}:=\wtil{Q}(\PP^n)$. The decomposition \eqref{qdecomp} of $\wtil{Q}_p$ then gives rise to the corresponding one of $\wtil{Q}$,
\begin{align}\label{QDEcomp}
\wtil{Q}\ =\ \wtil{Q}^{red}\ \cup\ \wtil{Q}^{(1)}\ \cup\ \wtil{Q}^{(2)}\ \cup\ \wtil{Q}^{(3)}.
\end{align}
Then the components $\wtil{Q}^{(i)}$ are the images of the following node-identifying morphisms
\begin{enumerate}
\item $\til{\iota}_{1} : \overline{M}_{1,1} \times Q^{\mathrm{red}}_{1,1,d} \hra \wtil{Q}$,
\item $\til{\iota}_{2} : \overline{M}_{1,1} \times Q_{0,2,d} \times \overline{M}_{1,1} \stackrel{2:1}{\lra} \wtil{Q}$,
\item $\til{\iota}_{3} : \overline{M}_{1,2} \times \PP Q'_{0,2,d} \hra \wtil{Q}$.
\end{enumerate}
In (3), $\PP Q'_{0,2,d}$ denotes the projectivisation of $\LL_1\dual \oplus \LL_2\dual$, sum of dual tautological line bundles over the locus $Q'_{0,2,d}\subset Q_{0,2,d}$ where $\mathrm{ev}_1=\mathrm{ev}_2$. The following Remark explains why $\wtil{Q}^{(3)}$ is the image of $\til{\iota}_{3}$.

\begin{rema}\label{333}
In fact, $\wtil{Q}^{(3)}$ should be (the image of) projectivisation of the pullback of 
$$
N_{\fM_{1,2,0}\times\fM_{0,2,d}/\fM_{2,0,d}}\ \cong\ (\LL\dual_{1} \boxtimes \LL\dual_{1}) \oplus (\LL\dual_{2} \boxtimes \LL\dual_{2})
$$ 
on $\overline{M}_{1,2} \times Q'_{0,2,d}$ since $\wtil{Q}$ is the base change of the blowup. It is equal to $\overline{M}_{1,2} \times \PP Q'_{0,2,d}$ if $\LL_1\cong\LL_2$ on $\overline{M}_{1,2}$. In \cite[pp.1221--1222]{Zi08}, Zinger proved that the evaluation morphism of the Hodge bundle $\cH\to \LL_j$ on $\overline{M}_{1,2}$ maps isomorphic to 
$$
\cH\ \xrightarrow{\ \sim\ }\ \LL_j(-D)\ \hookrightarrow\ \LL_j,
$$ 
where $D= \overline{M}_{1,1}\times\overline{M}_{0,3} \hookrightarrow \overline{M}_{1,2}$ is a boundary divisor of a collision of the two marked points. Thus we have $\LL_1 \cong \cH(D) \cong \LL_2.$
\end{rema}

\medskip
As we have mentioned in Section \ref{Sec1}, local computation with \eqref{KuModel} tells us that the $i$-th $p$-field space $\wtil{Q}_p^{(i)}$ is a vector bundle over $\wtil{Q}^{(i)}$,
$$
\wtil{Q}_p^{(i)}\ \cong\ h^0 \l( \left. \left(\oplus_{j=1}^m R\pi_* \left(\cL^{-\ell_j}\otimes \omega_{\cC_{\wtil{Q}}} \right)\right)\right|_{\wtil{Q}^{(i)}} \r).
$$
To avoid a confusion, we denote it by $P^{(i)}$ when we consider it as a bundle, but use $\wtil{Q}_p^{(i)}$ for the space. So the pullback of $P^{(i)}$ on $\wtil{Q}_p^{(i)}$ is the tautological bundle. On $\wtil{Q}_p^{(i)}$, the obstruction bundle $E^{(i)}$ was defined in Section \ref{sect:conesinobs}. Unlike {\em over} $\wtil{\fM}^{line}$, the decomposition $E^{(i)}=E_1^{(i)}\oplus E_2^{(i)}$, where 
$$
E_1^{(i)} = h^{-1} \left.\l( \EE_{\wtil{Q} / \wtil{\fM}}|_{\wtil{Q}^{(i)}} \r)\right|^\vee_{\wtil{Q}_p^{(i)}}, \ \ E_2^{(i)}= R^1\pi_* \l( \oplus_{i=1}^m \cL^{\otimes -\ell_i} \otimes \omega_{\cC} \r)\cong \pi_* \l( \oplus_{i=1}^m \cL^{\otimes \ell_i}\r)^\vee,
$$
of the obstruction bundle {\em over} $\wtil{\fM}$ is not so obvious, but it is proven in \cite[Equation (3.15)]{LL22}.

From now on for simplicity, we denote the domain of the morphism $\tilde{\iota}_i$ by $\bQ^{(i)}$, by $\bQ_p^{(i)}$ the fiber product $\bQ^{(i)}\times_{Q^{(i)}}Q^{(i)}_p$ and by $\bP^{(i)}$ the pullback of $P^{(i)}$. Explicitly,
\begin{enumerate}
\item $\bQ^{(1)} = \bar{M}_{1,1} \times Q_{1,1,d}^\red $,
\item $\bQ^{(2)} = \bar{M}_{1,1} \times Q_{0,2,d} \times \bar{M}_{1,1} $,
\item $\bQ^{(3)} = \bar{M}_{1,2} \times  \PP Q'_{0,2,d}$,
\end{enumerate}
and the bundle $\bP^{(i)}$ is
\begin{enumerate}
\item $\bP^{(1)} = \cH\, \boxtimes\,  \oplus_{i=1}^m \mathrm{ev}^* \cO_{\PP^n}(-\ell_i)$, 
\item $\bP^{(2)} = \left(\cH \boxtimes \oplus_{i=1}^m \mathrm{ev}_1^* \cO_{\PP^n}(-\ell_i) \boxtimes \cO_{\bar{M}_{1,1}} \right)\ \bigoplus \ \left(\cO_{\bar{M}_{1,1}} \boxtimes \oplus_{i=1}^m \mathrm{ev}_2^* \cO_{\PP^n}(-\ell_i) \boxtimes \cH \right)$,
\item $\bP^{(1)} = \cH\, \boxtimes\,  \oplus_{i=1}^m \mathrm{ev}_1^* \cO_{\PP^n}(-\ell_i)$.
\end{enumerate}
Recall that in (3), $\mathrm{ev}_1=\mathrm{ev}_2$. We denote by 
$$
\til{\iota}_{p,i}\ :\ \bQ_p^{(i)}\ \longrightarrow\ \wtil{Q}_p^{(i)}. 
$$
the base change of the node-identifying morphism $\wtil{\iota}_i$. and let $\bE^{(i)} := \til{\iota}_{p,i}^* E^{(i)}$. Then the decomposition $\bE^{(i)}=\bE_1^{(i)}\oplus \bE_2^{(i)}$ is
\begin{enumerate}
\item $\bE_1^{(1)} = \cH^\vee \boxtimes \mathrm{ev}^*T_{\PP^n} $, $\bE_2^{(1)} = \cO_{\bar{M}_{1,1}}\boxtimes(\oplus_i \pi_*\cL^{\otimes \ell_i})\dual$,
\item $\bE_1^{(2)} = \left(\cH^\vee \boxtimes \mathrm{ev}_1^*T_{\PP^n}\boxtimes \cO_{\bar{M}_{1,1}}\right) \oplus  \left(\cO_{\bar{M}_{1,1}} \boxtimes \mathrm{ev}_2^*T_{\PP^n}\boxtimes \cH^\vee\right)$, 

\noindent $\bE_2^{(2)} = \cO_{\bar{M}_{1,1}} \boxtimes \l(\oplus_i \pi_*\cL^{\otimes \ell_i} \r)\dual \boxtimes \cO_{\bar{M}_{1,1}}$,
\item $\bE_1^{(3)} =  \cH^\vee \boxtimes \mathrm{ev}_1^*T_{\PP^n} $, $\bE_2^{(3)} = \cO_{\bar{M}_{1,2}}\boxtimes (\oplus_i \pi_*\cL^{\otimes \ell_i})\dual$.
\end{enumerate}

Consider the pullback cosection $\sigma^{(i)} : \bE^{(i)} \to \cO_{\bQ^{(i)}_p}$, and decompose it into 
$$
\sigma_1^{(i)} : \bE_1^{(i)} \to \cO_{\bQ^{(i)}_p}\ \text{ and }\ \sigma_2^{(i)} : \bE_2^{(i)} \to \cO_{\bQ^{(i)}_p}
$$ 
accordingly. Using these cosections, we can define Kiem-Li's cosection localised Gysin maps $e^{\KL}(\bE^{(i)})$ and $e^{\KL}(\bE^{(i)}_j)$. Letting $\bold{C}^{(i)} := \til{\iota}_{p,i}^* C^{(i)}$, the multiplicative property of $e^{\KL}$ \cite[Theorem 3.2]{Oh18} tells us that \eqref{a} becomes
\begin{align}\label{eq:ideal}
[\wtil{Q}_p^{(i)}]\virt &= \frac{1}{\deg(\til{\iota}_{p,i})}(\til{\iota}_{p,i})_*\left(e^{\KL}(\bE^{(i)})\cap [\bC^{(i)}]\right)  \\
& = \frac{1}{\deg(\til{\iota}_{p,i})}(\til{\iota}_{p,i})_*\left( e^{\KL}(\bE_1^{(i)}) \cap e^{\KL}(\bE_2^{(i)}) \cap [\bC^{(i)}] \right). \nonumber
\end{align}

Since the cosection $\sigma^{(i)}_2$ on $\bE^{(i)}_2\cong\oplus_i \pi_*\cL^{\otimes \ell_i}$ is defined by the (dual of) defining equation $f$ as on $V$ in Section \ref{reduced}, the cycle $e^{\KL}(\bE_2^{(i)})\cap [\bC^{(i)}]$ is supported on the space $\bE_1^{(i)}  \times_{Q(\PP^n)} Q(X)$. Then the local computation \eqref{COsect} shows that the restriction of the cosection $\sigma^{(i)}_1$ defining $e^{\KL}(\bE_1^{(i)})$ to this locus $\bE_1^{(i)}  \times_{Q(\PP^n)} Q(X)$ is induced by the surjection
$$
df\ :\ T_{\PP^n}|_X\ \surra\ \oplus_i \cO_{\PP^n}(\ell_i)|_X, 
$$
whose kernel is $\ker(df) = T_X$. On the locus, $df$ then defines a short exact sequence of bundles
\begin{align}\label{SEES}
0\ \longrightarrow\ \bK^{(i)}\ \longrightarrow\ \bE^{(i)}_1\ \longrightarrow\ (\bP^{(i)})^\vee \ \longrightarrow\ 0,
\end{align}
where
\begin{enumerate}
\item $\bK^{(1)} := \cH^\vee \boxtimes \mathrm{ev}^*T_X$,
\item $\bK^{(2)} := \left(\cH^\vee \boxtimes \mathrm{ev}_1^*T_X\boxtimes \cO_{\bar{M}_{1,1}}\right) \oplus  \left(\cO_{\bar{M}_{1,1}} \boxtimes \mathrm{ev}_2^*T_X\boxtimes \cH^\vee \right)$,
\item $\bK^{(3)} := \cH^\vee \boxtimes \mathrm{ev}_1^*T_X$.
\end{enumerate}
The tautological section of $\bP^{(i)}$ defines a cosection of $(\bP^{(i)})^\vee$. On $\bQ^{(i)}_p(X):=\bQ^{(i)}_p\times_{Q(\PP^n)} Q(X)$, the cosection $\sigma^{(i)}_1$ on $\bE^{(i)}_1$ factors through the pullback of this tautological cosection. Thus, again by the multiplicative property \cite[Theorem 3.2]{Oh18}, we have
\begin{align}\label{aa2}
e^{\KL}\l( \bE_1^{(i)}\r) \cap \l( e^{\KL}(\bE_2^{(i)}) \cap [\bC^{(i)}] \r) = e^{\mathrm{FM}}\l(\bK^{(i)}\r) \cap e^{\KL}\l( (\bP^{(i)} )\dual \r) \cap \l( e^{\KL}(\bE_2^{(i)}) \cap [\bC^{(i)}] \r),
\end{align}
where $e^{\mathrm{FM}}$ denotes the Fulton-MacPherson intersection homomorphism, or Gysin map.

In Sections \ref{DeformCone} and \ref{nomore}, we will explain the second and third equalities below, respectively. The rest equalities and notations are explained after the equations,
\begin{align}\label{eq:ideal2}
[\wtil{Q}_p^{(i)}]\virt &= \frac{1}{\deg(\til{\iota}_{p,i})}(\til{\iota}_{p,i})_* \l( e^{\mathrm{FM}}\l(\bK^{(i)}\r) \cap e^{\KL}\l( (\bP^{(i)} )\dual \r) \cap \l( e^{\KL}\l(\bE_2^{(i)}\r) \cap [\bC^{(i)}] \r) \r) \\ \nonumber
&= \frac{(-1)^{m\cdot i}}{\deg(\til{\iota}_{i})}(\til{\iota}_{i})_* \l( e^{\mathrm{FM}}\l(\bK^{(i)}\r) \cap e^{\KL}\l(  \bE_2^{(i)} \r) \cap \left[\bC^{(i)}|_{\bQ^{(i)}}\right]  \r) \\ \nonumber
&= \frac{(-1)^{m\cdot i}}{\deg(\til{\iota}_{i})}(\til{\iota}_{i})_* \l( e^{\mathrm{FM}}\l(\bK^{(i)}\r) \cap e^{\KL}\l(  \bE_2^{(i)} \r) \cap \left[\bC_{(i)}|_{\bQ^{(i)}}\right]  \r) \\ \nonumber
& = \frac{(-1)^{m\cdot i}}{\deg(\til{\iota}_{i})}(\til{\iota}_{i})_* \l( e^{\mathrm{FM}}\l(\bK^{(i)}\r) \cap \left[ \bC_{(i)}|_{\bQ^{(i)}(X)} \right]\virt \r) \\ \nonumber
& = \frac{(-1)^{m\cdot i}}{\deg(\til{\iota}_{i})}(\til{\iota}_{i})_* \l(e\l( \frac{\bK^{(i)}|_{\bQ^{(i)}(X)} }{\bC_{(i)}|_{\bQ^{(i)}(X)}  } \r) \cap [\bQ^{(i)}(X)]\virt\r).
\end{align}
The first equality is from \eqref{eq:ideal} and \eqref{aa2}. In the fourth equality, the cone $\bC_{(i)}$, the pullback of $C_{(i)}$ in \eqref{eq:coneEmbed}, is a bundle over $\bQ^{(i)}_p$ which is smooth. So its pullback $\bC_{(i)}|_{\bQ^{(i)}}$ is a bundle over $\bQ^{(i)}$. Mimicking Proposition \ref{prop4.1}, we prove $e^{\KL}(\bE_2^{(i)}) \cap [\bC_{(i)}|_{\bQ^{(i)}}]$ is the pullback cycle of
\begin{enumerate}
\item $(-1)^{d(\sum_i \ell_i)}e^{\mathrm{ref}}(V_{1,1,d}) \cap\l([\bar{M}_{1,1}] \times [Q_{1,1,d}^\red(\PP^n)]\r)$,
\item $(-1)^{d(\sum_i \ell_i)+m}e^{\mathrm{ref}}(V_{0,2,d}) \cap\l([\bar{M}_{1,1}] \times [Q_{0,2,d}(\PP^n)] \times [\bar{M}_{1,1}]\r)$,
\item $(-1)^{d(\sum_i \ell_i)}e^{\mathrm{ref}}(V_{0,2,d}) \cap\l([\bar{M}_{1,2}] \times [\PP Q'_{0,2,d}]\r)$\footnote{Here, $\rank V_{0,2,d}$ is $d(\sum_i \ell_i)$ although it is of genus $0$ because $\mathrm{ev}_1=\mathrm{ev}_2.$}.
\end{enumerate}
We denote the pullback cycle in $A_*(\bQ^{(i)})$ by $[\bQ^{(i)}(X)]\virt$ and that in $A_*(C_{(i)}|_{\bQ^{(i)}})$ by $\left[ \bC_{(i)}|_{\bQ^{(i)}(X)} \right]\virt$. The space $\bQ^{(i)}(X):=\bQ^{(i)}\times_{Q(\PP^n)}Q(X)$ is the support. The last equality comes from the fact that $ \bC_{(i)}|_{\bQ_p^{(i)}(X)}$ is contained in $\bK^{(i)}|_{\bQ_p^{(i)}(X)}$ by the cone reduction criterion \cite[Lemma 4.4]{KL13}.

The second equality holds if the cone $\bC^{(i)}$ is isomorphic to the product $\bC^{(i)}|_{\bQ^{(i)}}\times_{\bQ^{(i)}}\bQ^{(i)}_p$ by the property of the tautological bundles and sections,
$$
e^{\KL}\l( (\bP^{(i)})\dual \r)\cap [\bQ_p^{(i)}] = (-1)^{\rank ( \bP^{(i)} ) } e^{\refi}(\bP^{(i)})\cap[\bQ_p^{(i)}] = (-1)^{m \cdot i} [\bQ^{(i)}].
$$ 
We have to be careful when we use the commutativity 
$$
e^{\KL}((\bP^{(i)})\dual) \cap  e^{\KL}(\bE_2^{(i)})\ =\ e^{\KL}(\bE_2^{(i)})\cap e^{\KL}((\bP^{(i)} )\dual)
$$ 
since the sequence \eqref{SEES} is not defined on the entire space $\bQ^{(i)}_p$. But we can use it if $\bC^{(i)}$ is a product. In fact $\bC^{(i)}$ is not a product itself but we deform it to a product. We work this in Section \ref{DeformCone}.

We know $\bC_{(i)}\to\bC^{(i)}$ \eqref{eq:coneEmbed} is almost isomorphic. Then taking twistings by divisors after blowups gives an actual isomorphism which induces the third equality. This work is addressed in Section \ref{nomore}.

After we get \eqref{eq:ideal2}, we prove Theorem \ref{QLP} in Section \ref{sect:Thm2pf}. When $X$ is a Calabi-Yau $3$-fold we prove Theorem \ref{QLP1} in Section \ref{sect:CY}.

\subsection{Deformation of the cone}\label{DeformCone}
Consider the normal cone 
$$
C_{\bC^{(i)} \cap \bE_2^{(i)} / \bC^{(i)}}\ \hookrightarrow\ \bE^{(i)}
$$ 
which is a deformation of $\bC^{(i)}$ via deformation to the normal cone \cite[Chapter 5]{Fu}. A direct computation shows it is also contained in the kernel of the cosection \eqref{COsect}.

\begin{lemm} The cone $C_{\bC^{(i)} \cap \bE_2^{(i)} / \bC^{(i)}}$ has a component of a product
\begin{align}\label{eq:defcone}
\Def(\bC^{(i)}) := \left. 
C_{\bC^{(i)} \cap \bE_2^{(i)} / \bC^{(i)}} \right|_{\bQ_p^{(i)}}\ \cong\ \bC^{(i)}|_{\bQ^{(i)}}\times_{\bQ^{(i)}} \bQ^{(i)}_p.
\end{align}
Other components vanish after taken by $e^{\KL}(\bE^{(i)})$.
\end{lemm}
\begin{proof}
We prove this by using the local coordinate rings in Section \ref{virdecomp} obtained by the cut-out model \eqref{KuModel}. Recall from \eqref{CRS} that locally $C^{(i)}$ is Spec of
\begin{align*}
\frac{R\; [X_{1j},X_{2j},P_{1i},P_{2i}]}{\left(
\begin{array}{c}
x_{1k}X_{1l}-x_{1l}X_{1k},\ x_{1k}P_{1l}-p_{1l}X_{1k},\ p_{1k}P_{1l}-p_{1l}P_{1k}, \\
x_{2k}X_{2l}-x_{2l}X_{2k},\ x_{2k}P_{2l}-p_{2l}X_{2k},\ p_{2k}P_{2l}-p_{2l}P_{2k}
\end{array}
\right),}
\end{align*}
where $R=  B[x,p]\; /(c_1x_{1}, c_2x_{2},c_1p_{1}, c_2p_{2})$ is a local coordinate ring of $\wtil{Q}_p$. 

In a neighborhood of a point in $\wtil{Q}_p^{(1)}$ or $\wtil{Q}_p^{(3)}$, we have seen $c_1=1$ in Sections \ref{diagc} and \ref{desin}, hence $x_1=p_1=0$. Pulling back via the node-identifying morphism, $\bC^{(i)}$ is a component defined by $\{c_2=0\}$ and $\bC^{(i)} \cap \bE_2^{(i)} \subset \bC^{(i)}$ is defined by $\{ X_2 = 0 \}=\{ X_{21} = \dots = X_{2n} = 0 \}$. Introducing a partner variable $X'_2$ of $X_2$, the cone $C_{\bC^{(i)} \cap \bE_2^{(i)} / \bC^{(i)}}$ is Spec of
\begin{align*}
\frac{R/(c_2,x_1,p_1)\; [X_{1j},X'_{2j},P_{1i},P_{2i}]}{\left(
\begin{array}{c}
x_{2k}X'_{2l}-x_{2l}X'_{2k},\ x_{2k}P_{2l},\ p_{2k}P_{2l}-p_{2l}P_{2k}
\end{array}
\right).}
\end{align*}
Then it is the union of $\{x_2=0\}$ and $\{P_2=0\}$. We show the component $\{x_2=0\}$ vanishes by $e^{\KL}(\bE^{(i)})$. To do so it is enough to show that it vanishes by $e^{\KL}(\bE_1^{(i)})$ by \cite[Theorem 3.2]{Oh18}. We show this by degree reason. The cycle $e^{\KL}(\bE_1^{(i)})\cap \{x_2=0\}$ is of degree
$$
\dim B[x,p] - \rank \bE_1^{(i)}\ =\ \dim B[x,p]-n-1.
$$
On the other hand, $e^{\KL}(\bE_1^{(i)})\cap \{x_2=0\}$ is contained in the degeneracy locus of the cosection, a pairinig with $p_2$. It is contained in $R/(c_2,x,p)\; [X_{1},P_{1},P_{2}]$ which has dimension less than or equal to $\dim B[x,p]-n-2.$ Thus $e^{\KL}(\bE_1^{(i)})\cap \{x_2=0\}=0$. The component $\{P_2=0\}$ defines the cone \eqref{eq:defcone}. 

The cone $\bC^{(2)}$ is defined by $\{c_1=c_2=0\}$, and $\bC^{(2)}\cap\bE^{(2)}$ is $\{X_1=X_2=0\}$ in addition. Then it has $4$ components
$$
\{x_1=x_2=0\}\ \cup\ \{x_1=P_2=0\}\ \cup\ \{P_1=x_2=0\}\ \cup\ \{P_1=P_2=0\}.
$$
Similarly we can show the first three will be killed by $e^{\KL}(\bE^{(2)})$ by degree reason. Precisely the first one is killed by $e^{\KL}(\bE_1^{(2)})$, but for the second and third one, we need to decompose $\bE_1^{(2)}$ into two parts and use one for each. The fourth one is the cone \eqref{eq:defcone}.
\end{proof}

\subsection{Local freeness of cones}\label{nomore}
In this section we relate the vector bundle $\bC_{(i)}|_{\bQ^{(i)}}$ and cone $\bC^{(i)}|_{\bQ^{(i)}}$. We suppress the notation $|_{\bQ^{(i)}}$ throughout the Section. Locally this restriction is $p_1=p_2=0$.

Consider the morphism $\bC_{(i)} \to \bE^{(i)}$, pullback of \eqref{eq:coneEmbed}, locally described in \eqref{normal}, and its projection
\begin{align}\label{aa}
\bC_{(i)}\ \longrightarrow\ \bE^{(i)}_1.  
\end{align}
Locally we can check $\bC^{(i)}$ is contained entirely in $\bE^{(i)}_1$. Since $\bC_{(i)}\to \bC^{(i)}$ is isomorphic on a dense open space, the closure of the image of \eqref{aa} is $\bC^{(i)}$. For $i=1,3$, \eqref{aa} vanishes locally on $\{x_1=0\}$ as explained in Section \ref{sect:conesinobs}. Globally this vanishing locus is the pullback of the intersection $\wtil{Q}^{(i)} \cap \wtil{Q}^{\red}$ of components in \eqref{QDEcomp}. Consider the blow-up $b^{(i)} : \what{\bQ}^{(i)} \to \bQ^{(i)}$ along this locus and denote the exceptional divisor by $\bD^{(i)}$. Then the embedding \eqref{aa} pulls back to
$$
(b^{(i)*}\bC_{(i)})(\bD^{(i)})\ \cong\ b^{(i)*}\bC^{(i)}\ \hookrightarrow\ b^{(i)*}\bE^{(i)}_1.
$$

For $i=2$, one blowup on the vanishing locus $\{x_1=0\}\cup\{x_2=0\}$ is not enough since this locus is not smooth. So we take a blowup along the intersection $\{x_1=0\}\cap\{x_2=0\}$ first and then take another blowup along the proper transform of $\{x_1=0\}$ and $\{x_2=0\}$, which are disjoint.\footnote{Another candidate could be a blow up along $\{x_1=0\}$ first and then another blowup along the transform of $\{x_2=0\}$. But we take the one due to its advantage on the computation.} We denote by $b^{(2)}: \what{\bQ}^{(2)} \to \bQ^{(2)}$ the composition of blowup morphisms. Set $\bD^{(2)}_1$ to be the sum of exceptional divisors of the first blowup and the corresponding one of $\{x_1=0\}$ for the second blowup. Similarly we set $\bD^{(2)}_2$ to be the sum of exceptional divisors of the first blowup and the corresponding one of $\{x_2=0\}$ for the second blowup. Recall that $\bC_{(2)}$ is the pullback of the normal bundle $N_{\fM_{1,1} \times \fM_{0,2} \times \fM_{1,1} / \fM_{2,0} }$,
$$
\bC_{(2)}\ \cong\ (\LL_1\dual \otimes \LL\dual_{1} ) \oplus ( \LL_2\dual \otimes \LL_{2}\dual ).
$$ 
Then the embedding \eqref{aa} pulls back to
\begin{align*}
b^{(2)*}\l(\LL_1\dual \otimes \LL_{1}\dual \r)\l(\bD^{(2)}_1\r) \oplus b^{(2)*}\l( \LL_2\dual \otimes \LL_{2}\dual\r)\l(\bD^{(2)}_2\r)\ \cong\ b^{(2)*}\bC^{(2)} \ \hookrightarrow\ b^{(2)*}\bE_1^{(2)}.
\end{align*}

\subsection{Proof of Theorem \ref{QLP}}\label{sect:Thm2pf}
So from the third equality of \eqref{eq:ideal2}, the equalities actually hold on the blowup $\what{\bQ}^{(i)}$ with twistings by the exceptional divisors. Hence \eqref{eq:ideal2} is
$$
[\wtil{Q}_p^{(i)}]\virt \ =\ \frac{(-1)^{m\cdot i}}{\deg(\wtil{\iota}_{i})}(\wtil{\iota}_{i})_* b^{(i)}_* \left( e\left(\frac{  \bK^{(i)}}{\bC_{(i)}(\bD^{(i)})}\right)\cap [ \what{\bQ}^{(i)}(X)]\virt\right),
$$
where $[\what{\bQ}^{(i)}(X)]\virt$ is the cycle, pushing down to $[\bQ^{(i)}(X)]\virt$ via the blowup morphism $b^{(i)}$. For $i=2$, we use
$$
\bC_{(i)}(\bD^{(i)})\ :=\ b^{(2)*}\l(\LL_1\dual \otimes \LL_{1}\dual \r)\l(\bD^{(2)}_1\r) \oplus b^{(2)*}\l( \LL_2\dual \otimes \LL_{2}\dual\r)\l(\bD^{(2)}_2\r)
$$
for notational consistence. Here, we could through away $\bD^{(i)}$ in the denominators by using \cite[Lemma 4.1]{LO20},
\begin{align}\label{eq:chernexpress1}
[\wtil{Q}_p^{(i)}]\virt \ =\ \frac{(-1)^{m\cdot i}}{\deg(\wtil{\iota}_{i})}(\wtil{\iota}_{i})_*  \left( \left[ \frac{ c( \bK^{(i)}) }{c(\bC_{(i)})} \right]_{\star} \cap [\bQ^{(i)}(X)]\virt\right),
\end{align} 
where $\star=\dim X-1$ for $i=1,3$ and $\star=2\dim X-2$ for $i=2$.

\medskip
We compute \eqref{eq:chernexpress1} explicitly to get Theorem \ref{QLP}.
\subsubsection{$i=1$ case} Recall from Section \ref{OUTLINE} that 
\begin{itemize}
\item $\bK^{(1)} = \cH^\vee \boxtimes \mathrm{ev}^*T_X$,
\item $\bC_{(1)} \cong \LL^\vee \boxtimes \LL^\vee$,
\item $[\bQ^{(1)}(X)]\virt = (-1)^{d(\sum_i \ell_i)}e^{\mathrm{ref}}(V_{1,1,d}) \cap\l([\bar{M}_{1,1}] \times [Q_{1,1,d}^\red(\PP^n)]\r)$.
\end{itemize}
Combining these with \cite[Theorem 1.1]{LL22}
\begin{align*}
e^{\mathrm{ref}}(V_{1,1,d}) \cap [Q_{1,1,d}^\red(\PP^n)] \ =\ [Q_{1,1,d}(X)]\virt - [K]_{\dim X-1}\cap 
\l( \, [\bar{M}_{1,1}] \times [Q_{0,2,d}(X)]\virt \, \r),
\end{align*} 
\eqref{eq:chernexpress1} for $i=1$ becomes
\begin{align}\label{FINAL1}
[Q^{(1)}_p]\virt  = & (-1)^{d(\sum_i \ell_i) + m}\, [K]_{\dim X-1} \cap  \left( [\overline{M}_{1,1}] \times [Q_{1,1,d}(X)]\virt \right) \\ \nonumber
& - (-1)^{d(\sum_i \ell_i) + m} \,[K_1]_{\dim X-1}[K_2]_{\dim X-1} \cap \l( \, [\overline{M}_{1,1}] \times  [Q_{0,2,d}(X)]\virt \times [\overline{M}_{1,1}] \, \r),
\end{align}
where, as introduced in Introduction, $K$ denotes the cohomology class $K= \frac{c\,(\cH^\vee \boxtimes \;\mathrm{ev}^* T_X)}{c\,(\LL^\vee \boxtimes\; \LL^\vee)}$.

\subsubsection{$i=2$ case} Recall from Section \ref{OUTLINE} that 
\begin{itemize}
\item $\bK^{(2)} = \left(\cH^\vee \boxtimes \mathrm{ev}_1^*T_X\boxtimes \cO_{\bar{M}_{1,1}}\right) \oplus  \left(\cO_{\bar{M}_{1,1}} \boxtimes \mathrm{ev}_2^*T_X\boxtimes \cH^\vee \right)$,
\item $\bC_{(2)} \cong (\LL_1\dual \otimes \LL\dual_{1} ) \oplus ( \LL_2\dual \otimes \LL_{2}\dual )$,
\item $[\bQ^{(2)}(X)]\virt = (-1)^{d(\sum_i \ell_i)+m}e^{\mathrm{ref}}(V_{0,2,d}) \cap\l([\bar{M}_{1,1}] \times [Q_{0,2,d}(\PP^n)] \times [\bar{M}_{1,1}]\r)$.
\end{itemize}
So for $i=2$, \eqref{eq:chernexpress1} becomes
\begin{align}\label{FINAL2}
[Q^{(2)}_p]\virt  = \frac{(-1)^{d(\sum_i \ell_i) + m}}{2}  \,[K_1K_2]_{2\dim X-2} \cap \l( \, [\overline{M}_{1,1}] \times  [Q_{0,2,d}(X)]\virt \times [\overline{M}_{1,1}] \, \r).
\end{align}

\subsubsection{$i=3$ case} Recall from Section \ref{OUTLINE} that 
\begin{itemize}
\item $\bK^{(3)} = \cH^\vee \boxtimes \mathrm{ev}_1^*T_X$,
\item $\bC_{(3)} \cong \LL^\vee \boxtimes \cO_{\PP Q'_{0,2,d}}(-1)$,
\item $[\bQ^{(3)}(X)]\virt = (-1)^{d(\sum_i \ell_i)}e^{\mathrm{ref}}(V_{0,2,d}) \cap\l([\bar{M}_{1,2}] \times [\PP Q'_{0,2,d}]\r)$
\end{itemize}
where $\cO_{\PP Q'_{0,2,d}}(-1)$ is the tautological line bundle of $\PP Q'_{0,2,d} = \PP(\LL\dual_{1}\oplus \LL\dual_{2})$. 

To compute \eqref{eq:chernexpress1} we first expand $c\,(\bK^{(3)})/c\,(\bC_{(3)})$
\begin{align*}
\frac{c\,(\bK^{(3)})}{1+ c_1(\LL^\vee)+c_1( \cO_{\PP Q'_{0,2,d}}(-1))}\ &=\ c\,(\bK^{(3)})\cdot \sum_{a\geq 0}
\frac{(-1)^a\cdot c_1(\cO_{\PP Q'_{0,2,d}}(-1))^a}{(1+ c_1(\LL^\vee))^{a+1}}\\
&=\ \sum_{a\geq 0}(-1)^a\cdot A^{a+1}\cdot c_1(\cO_{\PP Q'_{0,2,d}}(-1))^a,
\end{align*}
where $A^t=\frac{c\,(\cH^\vee \boxtimes\; \mathrm{ev}_1^*T_X)}{c\,(\LL^\vee\boxtimes\; 1)^{t}}$ as introduced in Introduction. Its $(\dim X-1)$-part $[c\,(\bK^{(3)})/c\,(\bC_{(3)})]_{\dim X-1}$ is 
\begin{align}\label{AA}
\sum_{a\geq 0}^{\dim X-1}(-1)^a\cdot [A^{a+1}]_{\dim X-1-a}\cdot c_1(\cO_{\PP Q'_{0,2,d}}(-1))^a.
\end{align}
By definition of Segre classes \cite[Chapter 3.1]{Fu}, we have
$$
p_* \l(c_1(\cO_{\PP Q'_{0,2,d}}(-1))^a \cap [\PP Q'_{0,2,d}]\r)\ =\ s_{a-1}\l( \LL\dual_{1}\oplus \LL\dual_{2} \r) \cap [Q'_{0,2,d}(\PP^n)]\ =\ [B_1B_2]_{a-1}\cap [Q'_{0,2,d}(\PP^n)]
$$
where $p: \PP Q'_{0,2,d}\to Q'_{0,2,d}(\PP^n)$ is the projection morphism and $B=\frac{1}{c\,(\LL^\vee)}$. So by the projection formula, capping \eqref{AA} with $[\bar{M}_{1,2}] \times [\PP Q'_{0,2,d}]$ and pushing it down to $\bar{M}_{1,2} \times Q'_{0,2,d}(\PP^n)$ becomes
\begin{align}\label{AAA1}
&p_*\l(\left[\frac{c\,(\bK^{(3)})}{c\,(\bC_{(3)})}\right]_{\dim X-1}\cap([\bar{M}_{1,2}] \times [\PP Q'_{0,2,d}])\r)\\ \nonumber
&=\sum_{a\geq 0}^{\dim X-1}(-1)^a\cdot [A^{a+1}]_{\dim X-1-a} [B_1B_2]_{a-1}\cap \l([\bar{M}_{1,2}] \times [Q'_{0,2,d}(\PP^n)]\r).
\end{align}

Next we compute the cycle $e^{\mathrm{ref}}(V_{0,2,d}) \cap\l([\bar{M}_{1,2}] \times [Q'_{0,2,d}(\PP^n)]\r)$ in $\bar{M}_{1,2}\times Q_{0,2,d}(\PP^n)$. Denoting by $j: Q'_{0,2,d}(\PP^n) \hookrightarrow Q_{0,2,d}(\PP^n)$ the embedding and by $\overline{V}_{0,2,d}$ the bundle $\oplus_{i=1}^{m} \pi_* \cL^{\ell_i}$ on $Q_{0,2,d}(\PP^n)$, the evaluation morphism gives rise to a sequence
$$
0\ \longrightarrow\ V_{0,2,d} \ \longrightarrow\ j^*\overline{V}_{0,2,d} \ \xrightarrow{\ \mathrm{ev}_1-\,\mathrm{ev}_2\ }\ \mathrm{ev}_1^*\oplus_{i=1}^{m} \cO(\ell_i) \ \longrightarrow\ 0.
$$
Denoting by $\Delta_{\PP^n}\in H^n(\PP^n \times \PP^n)$ the diagonal class, we have
\begin{align}\label{AAA2}
e^{\mathrm{ref}}(V_{0,2,d}) \cap\l([\bar{M}_{1,2}] \times [Q'_{0,2,d}(\PP^n)]\r)\ &=\ \frac{e^{\mathrm{ref}}(\overline{V}_{0,2,d})}{e(\oplus_{i=1}^{m} \cO(\ell_i))}\cap \l([\bar{M}_{1,2}] \times\l((\mathrm{ev}_1\times\mathrm{ev}_2)^*\Delta_{\PP^n}\cap [Q_{0,2,d}(\PP^n)]\r)\r)  \nonumber \\
&=\ [\bar{M}_{1,2}] \times\l(\frac{(\mathrm{ev}_1\times\mathrm{ev}_2)^*\Delta_{\PP^n}}{e(\oplus_{i=1}^{m} \cO(\ell_i))} \cap [Q_{0,2,d}(X)]\virt\r) \\
&=\ [\bar{M}_{1,2}] \times [Q'_{0,2,d}(X)]\virt \nonumber
\end{align}
where $[Q'_{0,2,d}(X)]\virt$ is the cycle defined in \eqref{Q1'}. Note that $\Delta_{\PP^n}|_X=e(T_{\PP^n}|_X)$ and $\Delta_X|_X=e(T_X)$. Hence by \eqref{AAA1} and \eqref{AAA2}, \eqref{eq:chernexpress1} becomes
\begin{align}\label{FINAL3}
[Q^{(3)}_p]\virt  = (-1)^{d(\sum_i \ell_i) + m} \sum_{a\geq 0}^{\dim X-1}(-1)^a\cdot [A^{a+1}]_{\dim X-1-a} [B_1B_2]_{a-1}\cap \l([\bar{M}_{1,2}] \times [Q'_{0,2,d}(X)]\virt\r).
\end{align}
So \eqref{FINAL1}, \eqref{FINAL2}, \eqref{FINAL3} and \eqref{X=p} prove Theorem \ref{QLP}.

\subsection{Calabi-Yau $3$-folds} \label{sect:CY}
Suppose that $X$ is a Calabi-Yau $3$-fold. Set
$$
\alpha := c_1(\cH^\vee\boxtimes 1), \ \  \beta := c_2(1\boxtimes \mathrm{ev}^*T_X) ,\ \ \psi := c_1(1\boxtimes \LL).
$$
\subsubsection{$i=1$ case}Then since $\alpha = c_1(\LL^\vee \boxtimes 1)$ we have
\begin{align*}
[K]_2\ =\ \left[\frac{c\,(\cH^\vee \boxtimes \mathrm{ev}^*T_{X})}{c\,(\LL^\vee \boxtimes \LL\dual )}\right]_{2} \ = \ \left[ \frac{(1 + 3\alpha + \beta) }{ (1 + \alpha - \psi) } \right]_2.
\end{align*}
Its nontrivial contribution to the integration over $[\overline{M}_{1,1}] \times (e^{\mathrm{ref}}(V_{1,1,d}) \cap [Q_{1,1,d}^\red(\PP^n)])$ is only $-\alpha\psi$. Hence 
$$
[Q^{(1)}_p]\virt\ =\ -\frac{(-1)^{d(\sum_i \ell_i) + m}}{24}\; c_1(\LL)\cap (e^{\mathrm{ref}}(V_{1,1,d}) \cap [Q_{1,1,d}^\red(\PP^n)])
$$
Using \cite[Corollary 1.3]{LL22}
\begin{align*}
e^{\mathrm{ref}}(V_{1,1,d}) \cap [Q_{1,1,d}^\red(\PP^n)] = [Q_{1,1,d}(X)]\virt  - \frac{c\,(\LL)}{12}[Q_{0,2,d}(X)]\virt , 
\end{align*}
we obtain
\begin{align}\label{AAAA1}
[Q^{(1)}_p]\virt\ =\ -\frac{(-1)^{d(\sum_i \ell_i) + m}}{24}\; c_1(\LL)\cap [Q_{1,1,d}(X)]\virt + 2\frac{(-1)^{d(\sum_i \ell_i) + m}}{24^2}c_1(\LL_1)c_1(\LL_2)\cap [Q_{0,2,d}(X)]\virt.
\end{align}

\subsubsection{$i=2$ case} Similarly we have
$$
\left[K_1K_2\right]_{4}  = \left[ \frac{(1 + 3\alpha_1 + \beta_1)}{(1 + \alpha_1 - \psi_1)} \cdot \frac{(1 + 3\alpha_2 + \beta_2)}{(1 + \alpha_2 - \psi_2)} \right]_4.
$$
The nontrivial contribution is $\alpha_1\alpha_2 (-3 \psi_1\psi_2 - 3\beta_1 - 3\beta_2)$. Hence we obtain
\begin{align}\label{AAAA2}
[Q^{(2)}_p]\virt\ &=\ \frac{(-1)^{d(\sum_i \ell_i) + m} }{2} \alpha_1\alpha_2 (-3 \psi_1\psi_2 - 3\beta_1 - 3\beta_2)  \cap \left( [\overline{M}_{1,1}] \times [Q_{0,2,d}(X)]\virt \times [\overline{M}_{1,1}] \right) \\ \nonumber
&=\ -(-1)^{d(\sum_i \ell_i) + m} \frac{3}{2\cdot 24^2} (c_1(\LL_1)c_1(\LL_2)+c_2(\mathrm{ev}_1^*T_X)+c_2(\mathrm{ev}_2^*T_X))  \cap  [Q_{0,2,d}(X)]\virt 
\end{align}

\subsubsection{$i=3$ case}
Since $(\mathrm{ev}_1  \times \mathrm{ev}_2)^*(\Delta_X) \in H^3(Q_{0,2,d}(X))$ and the degree of $[Q_{0,2,d}(X)]\virt$ is $2$, $[Q'_{0,2,d}(X)]\virt=0$. Thus $[Q_p^{(3)}]\virt = 0$.

\medskip

By \eqref{AAAA1}, \eqref{AAAA2} and \eqref{X=p}, we prove Theorem \ref{QLP1}.

\end{document}